\documentclass[11pt,reqno]{amsart}
\usepackage{amsmath, amsfonts, amssymb}
\usepackage[margin=3cm]{geometry}
\usepackage[dvipsnames, svgnames]{xcolor}
\usepackage{asymptote}
\usepackage{adjustbox, amssymb, enumitem, tikz-cd, multirow}
\usepackage{mark-macros}

\usepackage[pdfusetitle, pdfencoding=auto, psdextra, breaklinks=true]{hyperref}
\usepackage{xcolor}

\allowdisplaybreaks
\renewcommand{\textbf}[1]{{\bfseries\boldmath #1}}
\newcommand{\vocab}[1]{\textbf{\textcolor{BrickRed}{\boldmath #1}}}
\newcommand{\mathvocab}[1]{\textcolor{BrickRed}{#1}}

\definecolor{mylinkcolor}{rgb}{0.0,0.0,0.7}
\definecolor{myurlcolor}{rgb}{0.0,0.0,0.7}
\hypersetup{
 colorlinks,
 urlcolor=myurlcolor,
 citecolor=myurlcolor,
 linkcolor=mylinkcolor,
 breaklinks=true
}

\usepackage{thmtools, thm-restate}

\usepackage[nameinlink]{cleveref}

[name=Theorem]
[name=Claim, sibling=theorem]
\declaretheorem{lemma}[name=Lemma, sibling=theorem]
\declaretheorem{proposition}[name=Proposition, sibling=theorem]
\declaretheorem{observation}[name=Observation, sibling=theorem]
\declaretheorem{corollary}[name=Corollary, sibling=theorem]
\declaretheorem{definition}[style=definition, name=Definition, sibling=theorem]
\declaretheorem{remark}[style=definition, name=Remark, sibling=theorem]
[style=definition, name=Example, sibling=theorem]
[style=definition, name=Open Problem, sibling=theorem]
\declaretheorem{conjecture}[style=definition, name=Conjecture, sibling=theorem]

\usepackage[backend=biber, style=alphabetic, sorting=nyt,
maxnames=20, maxalphanames=20, backref, url=true, doi=true]{biblatex}
\DeclareSourcemap{
  \maps[datatype=bibtex]{
    \map[overwrite]{
      \step[fieldsource=doi, final]
      \step[fieldset=url, null]
      \step[fieldset=eprint, null]
    }  
  }
}
\DeclareSourcemap{
  \maps[datatype=bibtex]{
    \map{
      \step[fieldset=issn, null]
    }
  }
}
\DefineBibliographyStrings{english}{%
  backrefpage = {$\uparrow$ p.},%
  backrefpages = {$\uparrow$ pp.}%
}
\defbibheading{bibliography}[\refname]{\section*{#1}}

\title{Extremal $t$-intersecting Families of Permutations
for Large $t$}
\author{Pitchayut Saengrungkongka}

\address{
  Department of Mathematics,
  Massachusetts Institute of Technology,
  Cambridge,
  MA 02139,
  USA
}
\email{pitchayutbcc168@gmail.com}

\begin{document}
\begin{abstract}
A set of permutations of $\{1,2,\dots,n\}$ 
is $t$-intersecting 
if any two permutations agree on at least $t$ inputs.
A recent work by Kupavskii, 
in the spirit of the Erd\H os--Ko--Rado Theorem, 
shows that for all $t\leq n-O\left(\frac{n\log\log n}{\log n}\right)$, 
every $t$-intersecting family of permutations of $\{1,2,\dots,n\}$ 
with the maximum size must be isomorphic to the set
$$\mathcal A_k = \{\sigma : \sigma(i)=i\text{ for at least } t+k
\text{ indices } i\in\{1,2,\dots,t+2k\}\}$$
for some $k$. 
By refining Kupavskii's spread approximation technique, 
we prove that this conclusion holds for a wider range of 
$t\leq n-n^{5/7+\eps}$.
\end{abstract}
\maketitle

\section{Introduction}
One of the classical results in extremal set theory 
is the Erd\H os--Ko--Rado Theorem \cite{erdos_ko_rado}.
It states that for all positive integers $n$
and $r$ such that $n\geq 2r$ and for any family 
$\calF\subseteq \binom{[n]}r$ such that any 
two sets in $\calF$ have a nonempty intersection,
we have $|\calF| \leq \binom{n-1}{r-1}$.
In particular, the equality is achieved when $\calF$ consists 
of all sets containing $1$.
Since then, this result has been generalized 
in many different directions, leading to numerous 
variants of problems about intersecting families.

One way to generalize the Erd\H os--Ko--Rado Theorem 
is to consider $t$-intersecting families.
A family $\calF\subseteq \binom{[n]}r$ is said to be 
\vocab{$t$-intersecting} if any two elements $A,B\in\calF$ 
have $|A\cap B|\geq t$.
After a series of partial results 
by Frankl, Wilson, and Frankl and F\"uredi 
\cite{sets_partial_1, sets_partial_2, sets_partial_3}, 
Ahlswede and Khachatrian \cite{sets} completely determined 
the maximum size of a $t$-intersecting family of $\binom{[n]}r$.
In particular, they proved that any $t$-intersecting family 
$\calF\subseteq \binom{[n]}r$ with maximum size must be 
of the form 
$$\calA_k = 
\left\{A\in \binom{[n]}r : |A\cap [t+2k]|\geq t+k \right\}$$
for some $k$, up to permutation of elements.

Another direction of study is to replace $\binom{[n]}r$ 
with different families.
One of the most well-studied families is the family of
permutations of $[n]$.
Let $\Sigma_n$ be the collection of permutations of $[n]$.
Two permutations $\pi,\sigma\in \Sigma_n$
\vocab{intersect in at least $t$ elements}
if $\pi(i) = \sigma(i)$ for at least $t$ values of $i$.
A family $\calF\subseteq \Sigma_n$ is \vocab{$t$-intersecting} 
if any two permutations in $\calF$ intersect 
in at least $t$ elements.
We are interested in determining the maximum
size of a $t$-intersecting family $\calF\subseteq\Sigma_n$.

One class of $t$-intersecting families is the family 
$$\mathvocab{\calA_k} 
= \{\sigma\in\Sigma_n : \sigma(i)=i 
\text{ for at least } t+k\text{ indices }i\in[t+2k]\}.$$
Note that $\calA_0$ is the trivial family 
consisting of all permutations that fix $1,2,\dots,t$
(so it has size $(n-t)!$).
Analogous to Ahlswede and Khachatrian's theorem,
we have the following conjecture. 
\begin{conjecture}
\label{conj:main}
If $\calF\subseteq\Sigma_n$ is a $t$-intersecting family,
then $|\calF|\leq \max_k |\calA_k|$.
Moreover, if $|\calF| = \max_k |\calA_k|$, then 
$\calF = \sigma \calA_k\tau$ for some $k$ 
and permutations $\sigma$ and $\tau$.
\end{conjecture}
In particular, by comparing $|\calA_k|$ 
across different $k$, the conjecture implies that 
$\calA_0$ is the maximum family for all $t\leq \frac n2$.

The most recent progress towards resolving this 
conjecture is due to Kupavskii, who showed 
\cite[Thm.~1]{one_minus_eps}  that the conjecture holds for 
$t\leq (1-o(1))n$, together with a sharp stability result.
Upon closer inspection (cf. \Cref{rem:loglog}),
his result holds for all
$t\leq n - O\left(\frac{n\log\log n}{\log n}\right)$.
By enhancing this result with several new ideas,
namely the iterative spread approximation 
and a more delicate counting argument,
we improve the range of $t$ in which this conjecture 
holds while maintaining the same stability result.
\begin{restatable}{theorem}{main}
\label{thm:not_too_close}
For any $\eps > 0$, there exists $n_0=n_0(\eps)$
such that for any $n>n_0$ and $t< n-n^{5/7+\eps}$,
both of the following hold.
\begin{enumerate}[label=(\alph*)]
\item 
If $\calF\subseteq\Sigma_n$ is a $t$-intersecting family,
then $|\calF| \leq \max_k |\calA_k|$.
\item 
If $\calF\subseteq\Sigma_n$ is a $t$-intersecting family 
not contained in $\sigma\calA_k\tau$
for any $\sigma,\tau\in\Sigma_n$,
then $|\calF| \leq \left(1-\tfrac 1e + o(1)\right) 
\max_k |\calA_k|$.
\end{enumerate}
\end{restatable} 

\subsection{Previous Work and Context}
\Cref{conj:main} has seen significant progress over 
the past 50 years.

The first progress towards proving \Cref{conj:main} was
due to Frankl and Deza \cite{very_close}, who
proved \Cref{conj:main} when 
$n$ is sufficiently large in terms of $n-t$.
Cameron and Ku, and Larose and Malvenuto 
\cite{permutations_1a, permutations_1b}
proved \Cref{conj:main} for $t=1$.

The next significant breakthrough for higher $t$ came when
Ellis, Friedgut, and Pilpel \cite{log_log} showed that 
\Cref{conj:main} holds for $t\leq c\log\log n$ for 
some constant $c>0$.
Their proof used Hoffman's bound 
(upper bound on the size of an independent set of a graph) 
with weights chosen using representation theory 
of the symmetric group $S_n$.
For about ten years, there was no further improvement 
until Ellis and Lifshitz \cite{log}  
enhanced Ellis, Friedgut, and Pilpel's result 
using Fourier-analytic techniques and
showed that \Cref{conj:main} holds for all 
$t\leq c\frac{\log n}{\log\log n}$ for some constant $c>0$.

Kupavskii and Zakharov \cite{spread_approx}
made a breakthrough by showing that \Cref{conj:main}
holds for all $t\leq c\frac{n}{\log^2 n}$ for some constant $c>0$,
using a novel technique called the
\vocab{spread approximation method}.
This method is purely combinatorial and is based on 
a previous breakthrough in the sunflower problem 
by Alweiss, Lovett, Wu, and Zhang \cite{sunflower}.
It has been used to give results on variants of Erd\H os--Ko--Rado 
problems in various settings, such as 
intersecting families of partitions 
\cite{partitions}, intersecting families of spanning trees 
\cite{trees,trees_2}, and the Hajnal--Rothschild problem
\cite{hajnal_rothschild}.

Independently of Kupavskii and Zakharov's proof
\cite{spread_approx},
Keller, Lifshitz, Minzer, and Sheinfeld \cite{On}
proved that \Cref{conj:main} holds for all $t\leq cn$ 
for some constant $c>0$ by using Fourier-analytic techniques,
more specifically hypercontractivity of global functions.
Their proof share some high-level ideals to \cite{spread_approx},
namely the use of some notion of pseudorandomness.
However, the technical tools used are completely different.

Finally, Kupavskii \cite{one_minus_eps}
refined the spread approximation technique to show that 
\Cref{conj:main} holds for all $t\leq (1-\eps)n$
and $n>n_0(\eps)$.
This the first result in which the maximal family 
is not necessarily $\calA_0$ or $\calA_{\lfloor (n-t)/2\rfloor}$,
but it can be $\calA_k$ for any $k$.
The numerous candidates for maximal families substantially 
complicate the analysis.

As we shall explain in the next subsection
(and with more detail in \Cref{rem:loglog}),
Kupavskii's proof requires $t$ to be at least 
$n-O\left(\frac{n\log \log n}{\log n}\right)$.
We overcome several obstacles in Kupavskii's proof 
and show that \Cref{conj:main} holds 
for all $t\geq n-n^{5/7+\eps}$.

\subsection{Proof Outline}
We now describe the high-level overview of 
the spread approximation method and explain 
our new contributions.
Broadly speaking, the two new ideas that allow us to get 
past the previous result are (i) an iterative version of 
spread approximation and (ii) a new analysis using 
$r$ sets instead of just two as in \cite{one_minus_eps}.
We now explain these two ideas in more detail.

In its most basic form (as in \cite{spread_approx}), the method of spread 
approximation works by approximating a family 
of permutations $\calF\subseteq \Sigma_n$ 
by a family of partial permutations $\calS$.
Informally, a partial permutation is a permutation 
where some entries are allowed to be undefined.
See \Cref{subsec:notation} for a precise definition.
Roughly speaking, family $\calS$ satisfies three properties:
\begin{itemize}
\item each element of $\calS$ has size
(i.e., number of defined entries) 
not much larger than $t$;
\item at least $(1-o(1))$-fraction of elements in $\calF$ 
extend an element in $\calS$; and
\item $\calS$ is $t$-intersecting.
\end{itemize}
The step of obtaining $\calS$ is called 
the \vocab{spread approximation step},
which relies on the ``spread lemma''
(\Cref{lem:spread}) introduced by
Alweiss, Lovett, Wu, and Zhang \cite{sunflower}
in the study of sunflower-free families, and later sharpened by 
Tao and Hu \cite{spread_lemma,spread_lemma_2}.

Next, we analyze $\calS$ further 
by the process of peeling (described precisely in 
\Cref{subsec:peeling}).
Informally, we alternate between 
\begin{itemize}
\item taking out partial permutations of maximum size; and
\item simplifying the family (by removing elements 
from partial permutations in the family) so that 
it remains $t$-intersecting.
\end{itemize}
The process partitions $\calS$ into layers, where the $k$-th 
layer consists of partial permutations with $t+k$ elements. 
Each layer has some nice properties (\Cref{prop:peeling_properties})
that allow us to count and upper-bound the number of elements.
Ultimately, one can show that most elements in $\calF$ 
contain the same fixed $t$-element partial permutation
(i.e., $\calF$ is very close to $\calA_0$).
This yields the so-called stability result.
Finally, a simple counting argument shows that 
$\calF$ must actually be contained in $\calA_0$.

When our family $\calF\subseteq \Sigma_n$ is too sparse
(which occurs when $t$ is closer to $n$),
the spread approximation step no longer works.
To overcome this, Kupavskii \cite{one_minus_eps}
introduced an extra step, in which one 
proves a density boost result (i.e., there exists a partial permutation 
that contains higher-than-average number of elements in $\calF$),
and used that to construct a better spread approximation.
Kupavskii's density boost result relied on the ability 
to get the initial spread approximation and perform the
peeling step, which,
as we shall explain in \Cref{rem:loglog},
inherently requires $t\geq n-\frac{C\log\log n}{\log n}$.

To overcome this issue, we introduce our idea (i) 
of doing the spread approximation iteratively.
In particular, we do not start by 
trying to get a spread approximation.
Instead, we begin by proving a density boost result.
This density boost result is very weak,
but it is enough to initiate a spread approximation,
which in turn gives a better density boost result.
Consequently, we alternate between finding a spread approximation 
and proving a density boost result.
Each round of alternation results in an incrementally better 
spread approximation and a better density boost,
and eventually, the density boost result is good enough 
to construct our final spread approximation.
Our new iterative spread approximation step
results in \Cref{thm:accurate_spread_approx},
which works whenever $t\leq n-n^{1/2+\eps}$.

In the subsequent peeling step, Kupavskii \cite{one_minus_eps}
had to work harder in order to get a maximum family 
that is not $\calA_0$.
After constructing the peeling, he chose $k$ such that 
the $k$-th layer is large. 
In this layer, one can show that there exist 
partial permutations $A$ and $B$ such that $|A\cap B| = t$
(in particular, is not strictly greater than $t$). 
Thus, $|A\cup B| = t+2k$, giving our $(t+2k)$-element 
partial permutation. 
Then by a counting argument, one can show that most 
partial permutations in the $k$-th layer 
are contained in $A\cup B$, recovering the stability result.
This argument works whenever $t\leq n-n^{3/4+\eps}$
(see \Cref{rem:34} for an explanation of why it 
only works up to the exponent $\frac 34$).
Thus, combining Kupavskii's argument 
with our new spread approximation 
step would give a proof of \Cref{conj:main} 
for $t\leq n-n^{3/4+\eps}$.

Improving the exponent past $\frac 34$
requires a more elaborate counting argument,
which motivates the idea (ii) of 
using $r$ partial permutations instead of only two 
partial permutations $A$ and $B$.
More specifically, we show that there exist partial permutations 
$A_1,\dots,A_r$ with certain nice properties
(\Cref{lem:good_A}).
Then by counting partial permutations that have 
intersection of size $t$ with each of $A_1$,
\dots, $A_r$, 
we show that most partial permutations are contained in 
$A_1\cup \dots\cup A_r$ (\Cref{cor:bound_T_k}). 
This recovers the stability result.
We believe that our argument could be further optimized 
to reduce the exponent (see \Cref{rem:optimizing}
for further discussion).
Also, see \Cref{rem:why_57} for a discussion 
of why the exponent $\frac 57$ was chosen.

\subsection{Organization of this paper}
This entire paper is devoted to proving 
\Cref{thm:not_too_close}.
\Cref{sec:preliminaries} contains some background,
including notation in \Cref{subsec:notation},
elementary bounds in \Cref{subsec:binom_bound}
and \Cref{subsec:bound_A_k}, and results 
about spread sets in \Cref{subsec:spread}.
\Cref{sec:spread_approx} executes the spread approximation 
step, developing the iterative spread approximation method.
The peeling step begins in \Cref{sec:peeling},
where we recall the peeling construction from previous papers
and prove the rough bound on layer size.
The crux of the peeling step lies in \Cref{sec:bounding},
where we prove the existence of nice $A_1,\dots,A_r$
and use them to bound the size of each layer.
Finally, \Cref{sec:conclusion} puts everything together
to prove \Cref{thm:not_too_close}.

\subsection{Acknowledgements}
This research was initiated during 
the University of Minnesota Duluth REU 
with support from Jane Street Capital, 
NSF Grant 2409861, 
and donations from Ray Sidney and Eric Wepsic. 
We thank Colin Defant and Joe Gallian
for such a wonderful opportunity. 
We also thank Noah Kravitz for helpful discussions 
and Maya Sankar for detailed feedback on this paper.
\section{Preliminaries}
\label{sec:preliminaries}
\subsection{Notation and Convention}
\label{subsec:notation}
We ignore all floors and ceilings when they are not 
important to the proof.

We use the standard asymptotic notations 
(most prominently $O$, $o$, $\sim$, $\ll$, and $\lesssim$) 
throughout the paper. We note the following:
\begin{itemize}
\item $f(n) \ll g(n)$, $g(n)\gg f(n)$, and $f(n)=o(g(n))$
are equivalent.
\item $f(n)\lesssim g(n)$, $g(n)\gtrsim f(n)$, and $f(n)=O(g(n))$
are equivalent.
\end{itemize}

We let $\mathvocab{[n]} := \{1,2,\dots,n\}$ be the standard $n$-element set.
For any set $X$, we let 
\begin{itemize}
\item $\mathvocab{2^X}$ be the set of all subsets of $X$,
\item $\mathvocab{\binom{X}k}$ 
be the set of all subsets of $X$ of size $k$, and
\item $\mathvocab{\binom{X}{\leq k}}$
be the set of all subsets of $X$ of size at most $k$.
\end{itemize}
Calligraphic letters ($\calA$, $\calB$, \dots)
will generally be used to denote families of sets.
As introduced in \cite{spread_approx}, we use the following notation in dealing with 
families of sets:
for a set $X$ and families $\calF$, $\calS$,
we let 
\begin{align*}
\mathvocab{\calF[X]} &:= \{F\in\calF : X\subseteq F\} \\
\mathvocab{\calF(X)} &:= \{F\setminus X : F\in\calF, X\subseteq F\}
= \{F\setminus X : F\in \calF[X]\} \\
\mathvocab{\calF[\calS]} &:= \bigcup_{A\in\calS} \calF[A].
\end{align*}
Note that $|\calF[X]| = |\calF(X)|$.

Rather than viewing $\Sigma_n$ as a collection of permutations,
we view $\mathvocab{\Sigma_n}$ as a collection of sets by having 
$\sigma : [n]\to [n]$ correspond to the
set $\{(1,\sigma(1)), (2,\sigma(2)), \dots, (n,\sigma(n))\}$.
Thus, $\Sigma_n \subseteq \binom{[n]^2}{n}$
(here, $[n]^2 = [n]\times [n]$)
and a $t$-intersecting family of permutations of $[n]$
corresponds to a $t$-intersecting family $\calF\subseteq\Sigma_n$,
where any two sets in $\calF$ have intersection size 
at least $t$.
A \vocab{partial permutation} 
is a subset of some element of $\Sigma_n$.
We note that for any $X\subseteq [n]^2$, we have 
$\Sigma_n(X) \leq (n-|X|)!$.

Throughout this paper, we make the following 
blanket assumption on $n$ and $t$.
\begin{equation}
\label{eq:n_t_cond}
\boxed{\text{Let }n\text{ and }t\text{ be
sufficiently large positive integers 
such that } n^{0.99} < t < n.}
\end{equation}
If $t<n^{0.99}$, then \Cref{thm:not_too_close} holds
due to a previous result such as \cite{spread_approx}.
We include this assumption to simplify some 
technicalities involved in the proof.

For the entire paper, we let $\mathvocab u := n-t$.

\subsection{Elementary Estimates}
\label{subsec:binom_bound}
In this section, we collect some elementary estimates 
that will be used later.
\begin{lemma}
\label{lem:binom_bound}
For any integers $n\geq k\geq 0$,
we have $\binom nk \leq (en/k)^k$.
\end{lemma}
\begin{proof}
Note that $\binom nk \leq \frac{n^k}{k!}$
and $k!\geq \exp\left(\int_1^k \log x\, dx\right) 
= \left(\tfrac ke\right)^k$ 
Combining these two estimates gives the lemma.
\end{proof}
\begin{lemma}
\label{lem:factorial_bound}
For any positive integers $n\geq k$, we have 
$\frac{n!}{k!} \geq \left(\frac ne\right)^{n-k}$.
\end{lemma}
\begin{proof}
We first claim that $(k+1)(k+2)\cdots n  
\geq n!^{\frac{n-k}n}$. To prove this, note that
\begin{align*}
\big((k+1)(k+2)\cdots n \big)^n 
&= \prod_{i=0}^{n(n-k)-1} 
\left(\left\lfloor\frac in\right\rfloor + k + 1\right) \\
&\geq \prod_{i=0}^{n(n-k)-1} 
\left(\left\lfloor\frac i{n-k}\right\rfloor + 1\right) 
= (1\cdot 2\cdots n)^{n-k},
\end{align*}
and the middle inequality 
$\big\lfloor \frac in\big\rfloor + k
\geq \big\lfloor \frac i{n-k}\big\rfloor$
follows from taking floor on both sides 
of the inequality $\frac in + k \geq \frac i{n-k}$,
which holds for all $i\leq n(n-k)$.

Finally, we have that
$$\frac{n!}{k!}\ =\ (k+1)(k+2)\cdots n 
\ \geq\ n!^{\frac{n-k}n} 
\ \geq\ \left(\frac ne\right)^{n-k},$$
where the last inequality comes from the well-known estimate
$n!\geq \left(\frac ne\right)^n$.
\end{proof}
\begin{lemma}
\label{lem:binom_bound_2}
If $a$, $b$ and $x$ are positive integers such that 
$a>b$ and $x\geq 2b$, then 
$$\left(\frac xb -1 \right)^{b-a} \binom xa 
\ <\ \binom xb\ <\ \left(\frac xa\right)^{b-a} \binom xa$$
\end{lemma}
\begin{proof}
Note that 
\begin{align*}
\frac{\binom xb}{\binom xa}
&= \left(\frac{x-a}{a+1}\right) 
\cdots \left(\frac{x-(b-2)}{b-1}\right) 
\left(\frac{x-(b-1)}b\right),
\end{align*}
which is clearly between 
$\left(\frac xb-1\right)^{b-a}$ and $\left(\frac xa\right)^{b-a}$.
\end{proof}
\begin{lemma}
\label{lem:close_binomial}
If $b = o(\sqrt a)$, then 
$$\binom{a+b}b = (1+o(1)) \binom ab.$$
In particular, repeatedly applying this lemma gives 
$\binom{a+kb}b = (1+o(1))\binom ab$ for all 
nonnegative integer $k$ (note that we treat $k$
is a constant here).
\end{lemma}
\begin{proof}
Note that 
\begin{align*}
\frac{\binom{a+b}b}{\binom ab} 
&= \frac{(a+b)(a+b-1)\cdots (a+1)}{a(a-1)(a-2)\dots (a-b+1)} \\
&\leq \left(1+\frac{b}{a-b+1}\right)^b 
< \exp\left(\frac{b^2}{a-b+1}\right) = 1+o(1). \qedhere
\end{align*}
\end{proof}
\subsection{Bounds on \texorpdfstring{$|\calA_k|$}{|A\_k|}}
\label{subsec:bound_A_k}
In our language of viewing permutations as elements in 
$\Sigma_n$, the optimal family $\calA_k$ 
for each $0\leq k\leq \frac{n-t}2$ can be written as 
$$\calA_k = \left\{F\in\Sigma_n : 
\begin{array}{c}
F\text{ contains a }(t+k)
\text{-element subset of } \\
\{(1,1),(2,2),\dots,(t+2k,t+2k)\}
\end{array}\right\}.$$
Throughout this paper, we write 
$$\max_j |\calA_j| := \max_{j=0}^
{\left\lfloor\frac{n-t}2\right\rfloor} |\calA_j|.$$

In this subsection, we provide elementary bounds on 
the size of $\calA_k$ that we will use in our proof.
The following \Cref{lem:bound_A_k}
appeared in \cite{one_minus_eps};
we provide the proof for completeness.
\begin{lemma}[{\cite[Lem.~5]{one_minus_eps}}]
\label{lem:bound_A_k} 
For any positive integers $n,t,k$, 
we have $|\calA_k| \leq \binom{t+2k}{k}(n-t-k)!$.
Furthermore, if $k=o(n-t-k)$, then 
$$|\calA_k| = (1+o(1)) \binom{t+2k}k (n-t-k)!$$
\end{lemma}
\begin{proof}
Let $T=\{(1,1),(2,2),\dots,(t+2k,t+2k)\}$.
To prove the upper bound, we note that 
each element of $\calA_k$ is counted at least once by 
picking a $(t+k)$-element subset of $T$ (there are 
$\binom{t+2k}{t+k}=\binom{t+2k}k$ such subsets)
and permuting the remaining $n-t-k$ indices.
Thus, we have $|\calA_k| \leq \binom{t+2k}k(n-t-k)!$.

To prove the lower bound, we have to account 
for the fact that an element might intersect $T$ 
in more than $t+k$ elements.
We claim that 
$$|\calA_k|\geq  \binom{t+2k}{t+k} (n-t-k)! 
- (t+k+1)\cdot  \binom{t+2k}{t+k+1} (n-t-k-1)!.$$
To see why, first note that $\binom{t+2k}{t+k+1}(n-t-k-1)!$
counts the number of ways to choose $(t+k+1)$-element 
subset of $T$ and permute the remaining $n-t-k-1$ indices.
In particular, elements that intersect $T$ in $j$ 
elements are counted $\binom j{t+k}$ times in the first term
and $(t+k+1)\binom j{t+k+1}$ times in the second term.
Since $\binom j{t+k} = \frac{t+k+1}{j-t-k} \binom j{t+k+1} 
\leq (t+k+1)\binom j{k+1}$ for all $j\geq t+k+1$, 
we get that each element is counted at most once
in the right hand side. 
Thus, we have that 
\begin{align*}|\calA_k|
&\geq  \binom{t+2k}{t+k} (n-t-k)! 
- (t+k+1)\cdot  \binom{t+2k}{t+k+1} (n-t-k-1)! \\
&\geq \binom{t+2k}{k}(n-t-k)! 
\left(1 - (t+k+1)\cdot \frac{k}{t+k+1}\cdot \frac 1{n-t-k}\right) \\
&= (1-o(1)) \binom{t+2k}k (n-t-k)!
\end{align*}
since $k=o(n-t-k)$.
\end{proof}
We also need the following result, which says that 
the expression $\binom tj (n-t-j)!$ appearing in the 
above lemma decays quickly as it deviates from the 
maximum value, and so the sum of them across all $j$
is not too big.
\begin{lemma}
    \label{lem:bound_sum_A}
Let $n$ and $t$ be positive integers,
and recall that $u=n-t$.
Assume that $u^2 \gg t \gg u$
and $u\geq 10$.
\begin{enumerate}[label=(\alph*)]
\item If $j_0$ is the index $j\in \left[0,\frac{u}2\right]$
such that 
$\binom tj (n-t-j)!$ is maximum, then $j_0=(1+o(1))\frac tu$.
\item There exists a constant $C>0$ such that 
$$\sum_{j=0}^{0.1u} \binom tj (n-t-j)! 
\leq C\sqrt{\frac tu} \cdot \max_{0\leq j\leq u} 
\binom tj (n-t-j)!.$$
\end{enumerate}
\end{lemma}
\begin{proof}
Let $u=n-t$ and let $f(j) = \binom tj (n-t-j)!$.
Thus, 
$$\frac{f(j)}{f(j+1)} = \frac{(j+1)(n-t-j)}{t-j}
= \frac{(j+1)(u-j)} {t-j}.$$
In particular, $f(j)/f(j+1)$ is increasing 
when $j\in [0,u/2]$. Moreover, 
\begin{itemize}
\item $f(j)/f(j+1) \leq (j+1) \frac ut$ for all $j\leq\frac u2$
(this uses $t\gg u$);
\item $f(j)/f(j+1) \geq \frac 12(j+1) \frac ut$ for all $j
\leq \frac u2$; and
\item $f(j)/f(j+1) = (1+o(1)) \frac ut(j+1)$ for all 
$j\ll u$.
\end{itemize}
Note that
$f(j_0)/f(j_0+1) \geq 1$ and $f(j_0-1)/f(j_0) \leq 1$.
From the second bullet point, we deduce that 
$j_0\leq \frac{2t}u$.
Since $u \gg \frac tu$, $j_0$ lies 
in the regime where the third bullet point applies,
and so we deduce that $j_0 = (1+o(1))\frac tu - 1
= (1+o(1)) \frac tu$,
proving (a).
(This uses $t\gg u$.)

Next, we note that for any $a,b<0.1u$, we have 
\begin{align*}
\left|\frac{f(a)}{f(a+1)} - \frac{f(b)}{f(b+1)}\right|
&= \left|\frac{(a+1)(u-a)}{t-a} 
- \frac{(b+1)(u-b)}{(t-b)}\right| \\
&= \left|\frac{(a-b)(tu+ab-at-bt-t+u)}{(t-a)(t-b)}\right| \\
&\geq |a-b| \cdot \frac{tu-0.1tu-0.1tu-0.1tu}{t^2}
\tag{using $a,b<0.1u$ and $u\geq 10$} \\
&> |a-b| \cdot \frac{u}{2t}.
\end{align*}
Therefore, we have that for all $j\leq 0.1u$
and $j-j_0\geq \sqrt{t/u}$, we have
\begin{align*}\frac{f(j)}{f(j+1)} 
\ \geq\ \frac{f(j_0)}{f(j_0+1)} 
+ (j-j_0) \cdot \frac{u}{2t} 
\ \geq\  1 + \sqrt{\frac{u}{4t}}
\end{align*}
In particular, if $j-j_0 \geq k\sqrt{t/u}$, then 
repeatedly applying this bound gives
\begin{align*}
\frac{f(j)}{f(j_0)}
\ =\ \prod_{i=j_0}^{j-1} \frac{f(i+1)}{f(i)} 
\ \leq\ \prod_{i=j_0+\sqrt{t/u}}^{j-1} \frac{f(i+1)}{f(i)} 
\ \leq\ \left(1+\sqrt{\frac{u}{4t}}\right)^{-(k-1)\sqrt{t/u}} 
\ \leq\ c^{k-1}
\end{align*}
for some constant $c<1$.
Analogously, one can pick $c$ larger and show that 
if $j-j_0 \leq -k\sqrt{t/u}$,
then $f(j) \leq c^{k-1} f(j_0)$.

We apply these bounds in the following fashion.
\begin{itemize}
\item[] \hspace{2.5cm} $\vdots$ 
\item $f(j) \leq c^2 f(j_0)$ for all $j-j_0\in 
[-4\sqrt{t/u}, -3\sqrt{t/u}]$.
\item $f(j) \leq c f(j_0)$ for all $j-j_0\in 
[-3 \sqrt{t/u}, -2\sqrt{t/u}]$.
\item $f(j) \leq f(j_0)$ for all 
$j-j_0\in [-2 \sqrt{t/u}, 2\sqrt{t/u}]$.
\item $f(j) \leq c f(j_0)$ for all $j-j_0\in 
[2\sqrt{t/u}, 3\sqrt{t/u}]$.
\item $f(j) \leq c^2 f(j_0)$ for all $j-j_0\in 
[3 \sqrt{t/u}, 4\sqrt{t/u}]$.
\item[] \hspace{2.5cm} $\vdots$
\end{itemize}
Summing these across all $j$ gives
$$\sum_{j=0}^{u^{0.99}} f(j) \leq 
\sqrt{\frac tu} (4+2c+2c^2+\dots ) f(j_0) \leq C
\sqrt{\frac tu} f(j_0),$$
where $C=4+2c+2c^2+\dots = \frac{4-2c}{1-c}$. 
This completes the proof of (b).
\end{proof}
\subsection{Spread Families}
\label{subsec:spread}
In this subsection, we recall some background about 
spread families and collect some useful facts 
that will be helpful in constructing a spread approximation 
in \Cref{sec:spread_approx}.
\begin{definition}
Let $N$ be a positive integer and $r>0$ be a real number.
A family $\calF\subseteq 2^{[N]}$ is said to be 
\vocab{$r$-spread} if for all $S\subseteq [N]$, we have 
$$|\calF[S]| \leq r^{-|S|}|\calF|.$$
(In other words, for any $S\subseteq [N]$, 
at most $r^{-|S|}$ fraction of sets in $\calF$ 
contains $S$.)

A family $\calF\subseteq 2^{[N]}$ is said to be 
\vocab{$(r,t)$-spread} if $\calF(T)$ is $r$-spread 
for all $t$-element subsets $T\subseteq [N]$.
\end{definition}

In this paper, we will always use this notion with 
$N=n^2$ and view $[N]$ as $[n]^2$
by relabeling the elements.

Spread families behave nicely in terms of 
applying the probabilistic method.
This observation was first discovered in a breakthrough 
on the sunflower problem
by Alweiss, Lovett, Wu, and Zhang \cite{sunflower},
and later sharpened by Tao \cite{spread_lemma}
and Hu \cite{spread_lemma_2}.
In the following lemma, a \vocab{$p$-random subset}
$W\subseteq [N]$ is a random subset obtained by 
including each element in $[N]$ independently with 
probability $p$.
\begin{lemma}[Spread Lemma]
\label{lem:spread}
Let $N$, $n$, and $m$ be a positive integer
such that $n\leq N$. 
Let $\calF\subseteq \binom{[N]}{\leq n}$ be an 
$r$-spread family for some positive real number $r$
and $W$ be a $(m\delta)$-random subset of $[N]$
for some real number $0<\delta<\frac 1m$.
Then,
$$\mathbb P(\text{there exists } F\in\calF
\text{ such that }F\subseteq W) 
\geq 1 - n\left(\frac{5}{\log_2(r\delta)}\right)^m.$$
\end{lemma}
\begin{proof}
See \cite[Prop.~2]{spread_lemma_2}
(cf. \cite[Prop.~5]{spread_lemma}).
\end{proof}

We now extract the corollary of \Cref{lem:spread}
that we will use in the paper.
\begin{corollary}
\label{cor:spread_lemma}
Let $n\leq N$ be positive integers and $r>0$.
Let $\calG_1,\calG_2\subseteq \binom{[N]}{\leq n}$ 
be two $r$-spread families.
If $r\geq 2^{12} \log_2(2n)$, then 
there exists $G_1\in\calG_1$ and $G_2\in\calG_2$
such that $G_1\cap G_2=\emptyset$.
\end{corollary}
\begin{proof}
Suppose that $W_1$ is a $\tfrac 12$-random 
subset of $[N]$, and $W_2 = [N]\setminus W_1$,
so $W_2$ is also a $\tfrac 12$-random subset of $[N]$
(but is not independent of $W_1$).
We apply the above \Cref{lem:spread} with parameters 
$m = \lceil \log_2 n\rceil \leq \log_2(2n)$
and $\delta = (2\lceil \log_2 n\rceil)^{-1}$.
In particular,
$$r\delta \geq \frac{2^{12}\log_2(2n)}{2\lceil\log_2 n\rceil}
\geq \frac{2^{12}\log_2(2n)}{2\log_2(2n)} \geq 2^{11}$$
so we have
\begin{align*}
\mathbb P(\text{there exists }G_1\in\calG_1 
\text{ such that }G_1\subseteq W_1) 
&> 1 - n\left(\frac{5}{\log_2(r\delta)}\right)^m  \\
&\geq 1 -  \left(\frac 5{11}\right)^{\lceil\log_2 n\rceil} 
> \frac 12.
\end{align*}
Similarly, we get that 
$$\mathbb P(\text{there exists }G_2\in\calG_2 
\text{ such that }G_2\subseteq W_2)  > \frac 12.$$
By union bound, we obtain that
with positive probability, 
these two events happen simultaneously.
In particular, there is a choice of $W_1$ and $W_2$
such that $G_1\subseteq W_1$ and $G_2\subseteq W_2$
for some $G_1\in\calG_1$ and $G_2\in\calG_2$.
Since $W_1$ and $W_2$ are defined to be disjoint,
we get that $G_1$ and $G_2$ are disjoint.
\end{proof}

The following observation indicates one way to find 
a spread family.
In particular, it explains how the conditions in 
\Cref{cor:spread_lemma} can be realized.
\begin{observation}
\label{obs:finding_spread}
For any family $\calF\subseteq 2^{[N]}$ and $r>0$, 
let $X\subseteq [N]$ be a set maximal under inclusion
that satisfies $|\calF(X)| \geq r^{-|X|}|\calF|$.
Then $|\calF(X)|$ is $r$-spread.
\end{observation}
\begin{proof}
We have to show that for any $Y\subseteq [N]\setminus X$,
we have $|\calF(X\cup Y)| \leq r^{-|Y|} |\calF(X)|$.
Assume that this is not true. Then we have 
$$|\calF(X\cup Y)| > r^{-|Y|} |\calF(X)| 
\geq r^{-|X\cup Y|}|\calF|,$$
which implies that $X\cup Y$ is a set larger than $X$
that satisfies the condition, contradicting the maximality of $X$.
\end{proof}
\section{Spread Approximation}
\label{sec:spread_approx}
The goal of this section is to prove the following theorem,
which allows us to replace $\calF$ with its \vocab{spread approximation}
$\calS$ that has a simpler structure.
Throughout this section, we assume that $n$
and $t$ are positive integers such that
$n-t > n^{1/2 + \eps}$, an assumption strictly stronger than 
the hypothesis of \Cref{thm:not_too_close}.
\begin{restatable}{theorem}{spreadapprox}
\label{thm:accurate_spread_approx}
Let $\eps \in \left(0,\frac 12\right)$. Then there exists $n_0=n_0(\eps)$
such that for any $n>n_0$ and $n^{0.99}<t<n-n^{1/2+\eps}$,
if $\calF\subseteq \Sigma_n$ is a $t$-intersecting family,
then there exists a family $\calS$ of partial permutations such that 
\begin{enumerate}[label=(\alph*)]
\item each element of $\calS$ is a partial permutation 
of size at most $t+0.001 u$;
\item $|\calF\setminus\calF[\calS]| < \frac 1n\cdot u!$;%
\footnote{Note that the choice of $\frac 1n$ in 
the condition (b) is not very essential;
it is to ensure that $|\calF\setminus \calF[\calS]|$
grows much slower than $u!$, which will make the application 
in \Cref{sec:conclusion} more convenient.} and
\item $\calS$ is $t$-intersecting.
\end{enumerate}
\end{restatable}

We prove this in two steps. 
The first step (done in \Cref{subsec:iterative}) is iterative. 
We alternate between finding a spread approximation 
and proving a density boost statement.
Spread approximation yields a density boost statement,
which is then used to get a spread approximation 
with smaller uniformity (i.e., the maximum size across all 
partial permutations in our $\calS$). 
We repeat until the uniformity reaches the target.
At this point, the resulting spread approximation 
will only be $t'$-intersecting 
for $t'$ slightly smaller than $t$.
In the second step (done in \Cref{subsec:accurate}),
we turn the final density boost statement into 
an accurate spread approximation that is $t$-intersecting.

\subsection{Iterative Spread Approximation}
\label{subsec:iterative}
The goal of this subsection is to prove the following 
lemma, written in the form that is amenable to induction.
Here, we ignore floors and ceilings.
\begin{lemma}
\label{lem:iterative}
Let $\delta \in \left(0,\tfrac 12\right)$ and $M>0$. 
There exists $n_0=n_0(\delta,M)$ such that for all $n>n_0$
and $n^{0.99} < t < n-n^{1/2+2\delta}$, the following assertions $A_i$ 
and $B_i$ hold for all nonnegative integers $i$ 
such that $(1-\delta)^i u > n^{1/2+2\delta}$.
\begin{enumerate}
\item[($A_i$)] (Spread approximation) 
For any $t$-intersecting family $\calF\subseteq \Sigma_n$,
there exists a family $\calS$ such that 
\begin{enumerate}[label=(\roman*)]
\item each element of $\calS$ is a partial permutation of size 
at most $q=t+(1-\delta)^i u$.
\item $|\calF\setminus\calF[\calS]| < \frac {2^i}{n^M}\cdot u!$.
\item $\calS$ is $t'$-intersecting,
where $t'=t - n^{1/2}$.
\end{enumerate}
\item[($B_i$)] (Density boost) 
For any $t$-intersecting $\calF\subseteq \Sigma_n$ 
such that $|\calF|\geq \frac {2^{i+1}}{n^M} \cdot u!$,
there exists $X\subseteq [n]^2$ such that
$|X| = t-n^{1/2}$ and 
$$|\calF(X)| \geq 
\frac 1{(10n^{1/2-\delta})^{(1-\delta)^{i+1} u}}
|\calF|.$$
\end{enumerate}
\end{lemma}
\begin{proof}
We use induction on $i$ as follows.
Note that $A_0$ is clear by taking $\calS=\calF$,
and we prove that 
(1) $A_i$ implies $B_i$ and (2) $B_i$ implies $A_{i+1}$.

\textbf{Proof that $A_i$ implies $B_i$.}
Let $\calF\subseteq\Sigma_n$ be $t$-intersecting 
and $|\calF| \geq \frac{2^{i+1}}{n^M}\cdot u!$.
Take $\calS$ as in the assertion $A_i$,
so $|\calF\setminus\calF[\calS]| < \frac{2^i}{n^M}\cdot u! 
< \frac 12|\calF|$, which means that 
$|\calF[\calS]| \geq \frac 12|\calF|$.%
\footnote{This part of the argument explains the factor 
$2^i$ in the bound $\frac{2^i}{n^M}\cdot u!$:
the required size of $\calF$ grows by a factor of $2$
at each step. The $n^M$ in the denominator is designed 
to absorb this $2^i$ when we apply this lemma.}

Fix any element $S\in \calS$.
Since $\calS$ is $t'$-intersecting,
any element in $\calF[\calS]$ must contain a $t'$-element subset 
of $S$.
Indeed, for any $F\in\calF[\calS]$, we have that 
$F$ contains some $T\in\calS$, and because 
$|F\cap S| \geq |T\cap S| = t'$, we see that $F$ contains 
some $t'$-element subset of $S$.

Since there are $\binom{|S|}{t'}< \binom{t+(1-\delta)^iu}{t'}$
possible $t'$-element subsets of $S$,
by the pigeonhole principle, one of them, 
say $X$, must be contained in a
$\binom{t+(1-\delta)^iu}{t'}^{-1}$-fraction 
of elements of $|\calF[\calS]|$.
In particular, there exists a set $X$ with size $t'$ such that
\begin{align*}
|\calF(X)| &\geq \binom{t+(1-\delta)^i u}{t'}^{-1} |\calF[\calS]| \\
&\geq \frac 12\binom{t+(1-\delta)^i u}{t-n^{1/2}}^{-1}
 |\calF|
\tag{since $|\calF[\calS]|\geq \tfrac 12|\calF|$} \\
&= \frac 12 \binom{t+(1-\delta)^i u}
{(1-\delta)^i u + n^{1/2}}^{-1} |\calF| \\
&\geq \frac 12\binom{n}{(1-\delta)^i u + n^{1/2}}^{-1} |\calF| 
\\
&> \frac 12\left(\frac{(1-\delta)^i u + n^{1/2}}{en}\right)^{(1-\delta)^i u + n^{1/2}} |\calF|
\tag{\Cref{lem:binom_bound}} \\
&> \frac 12\left(\frac 1{e\,n^{1/2-2\delta}}\right)^{(1-\delta)^i u + n^{1/2}} |\calF| 
\tag{since $(1-\delta)^i u > n^{1/2+2\delta}$}\\
&> \left(\frac 1{10n^{1/2-2\delta}}\right)^{(1-\delta)^i u + n^{1/2}} |\calF| \\
&> \frac 1{(10n^{1/2-\delta})^{(1-\delta)^{i+1} u}} |\calF|.
\end{align*}
To justify the last step, we note that 
$$\left(\frac 12 - \delta\right)(1-\delta)
> \left(\frac 12 - 2\delta\right)\left(1+\frac{n^{1/2}}
{(1-\delta)^i u}\right),$$
since $\left(\tfrac 12 - \delta\right)(1-\delta)
> \left(\tfrac 12 - 2\delta\right)$ for all $\delta>0$ 
and $\frac{n^{1/2}}{(1-\delta)^i u} < n^{-2\delta}
\to 0$ as $n\to\infty$.

This gives the desired $X$, proving $B_i$.

\medskip
\textbf{Proof that $B_i$ implies $A_{i+1}$.}
We do this through a standard spread approximation argument
(e.g., as in \cite[Lem.~10]{spread_approx})
with a little modification of having a density boost step.

Let $\calF$ be a $t$-intersecting family 
such that $|\calF| \geq \frac{2^{i+1}}{n^M}u!$.
Let $q=t+(1-\delta)^{i+1}u$ and $r=\frac{10 u}{\sqrt n}$. 
Initialize $\calF_1=\calF$. For each $j\geq 1$, do the following.
\begin{enumerate}
\item If $|\calF_j| < \frac{2^{i+1}}{n^M}\cdot u!$, then terminate.
\item Otherwise, we can apply the assertion 
$B_i$ on $\calF_j$ to get 
$X_j$ such that $|X_j|=t-n^{1/2}$ and 
\begin{equation}
\label{eq:density_boost_usage}
|\calF_j(X_j)| \geq  \frac 1{(10n^{1/2-\delta})^{(1-\delta)^{i+1} u}} |\calF_j|.
\end{equation}
\item Take maximal (under inclusion) $S_j\supseteq X_j$ 
such that 
\begin{equation}
\label{eq:spreadness_1}
|\calF_j(S_j)| \geq r^{-|S_j|+|X_j|}|\calF_j(X_j)|.
\end{equation}
Thus, by \Cref{obs:finding_spread} 
(applied in $\calG=\calF_j(X_j)$ to get set $Y$,
and then take $S_j = Y\cup X_j$),
we get that $\calG(S_j) = \calF_j(S_j)$ is $r$-spread.
\item If $|S_j| > q$, then terminate.
\item Otherwise, let $\calF_{j+1} = \calF_j\setminus\calF_j[S_j]$.
\end{enumerate}
The algorithm terminates when either $|S_j| > q$ 
or $|\calF_j| < \frac{2^{i+1}}{n^M}\cdot u!$.
Suppose this happens when $j=N+1$.
We take $\calS=\{S_1,\dots,S_N\}$. 
We claim that this choice of $\calS$ works. 
Note that (i) is clearly satisfied.
We now check (ii) and (iii).
\begin{enumerate}
\item[(ii)]
To prove (ii), we note that 
$\calF\setminus\calF[\calS]\subseteq \calF_{N+1}$
since we only remove elements contained in $S_i$ for some $i$.
It thus suffices to show that $|\calF_{N+1}|<\frac{2^{i+1}}{n^M}\cdot u!$.
If the algorithm terminates at step (1), 
then this is obviously true.
Otherwise, the algorithm terminates at step (4),
so $|S_{N+1}| > q$.
Let $d=|S_{N+1}| - |X_{N+1}|=q-t+n^{1/2}$,
so $d > q-|X_{N+1}| > (1-\delta)^{i+1} u$.
Then we compute
\begin{align*}
|\calF_{N+1}| 
&\leq (10n^{1/2-\delta})^{(1-\delta)^{i+1}u} 
|\calF_{N+1}[X_{N+1}]| \tag{from \eqref{eq:density_boost_usage}}\\
&\leq r^d (10n^{1/2-\delta})^{(1-\delta)^{i+1}u}  
|\calF_{N+1}[S_{N+1}]|  \tag{from \eqref{eq:spreadness_1}}\\
&\leq r^d (10n^{1/2-\delta})^{(1-\delta)^{i+1}u}  
(n-|S_{N+1}|)! \tag{since $|\calF_{N+1}[S_{N+1}]| 
\leq |\Sigma_n[S_{N+1}]| = (n-|S_{N+1}|)!$}\\
&= r^d (10n^{1/2-\delta})^{(1-\delta)^{i+1}u} 
(u+n^{1/2}-d)!  \\
&\leq r^d (10n^{1/2-\delta})^{(1-\delta)^{i+1}u} 
\frac{u!}{(u/e)^{d-n^{1/2}}}
\tag{using \Cref{lem:factorial_bound} and note $d-n^{1/2}>0$}\\
&\leq \left(\frac{e\,r}{u^{1-0.1\delta}}\right)^{d} 
(10n^{1/2-\delta})^{(1-\delta)^{i+1}u} 
\cdot u! \tag{since $d>(1-\delta)^{i+1}u >\frac{10}{\delta} n^{1/2}$}\\
&\leq \left(\frac{e\,r}{u^{1-0.1\delta}}\right)^{(1-\delta)^{i+1} u} 
(10n^{1/2-\delta})^{(1-\delta)^{i+1}u} 
\cdot u! \qquad
\tag{since $d>(1-\delta)^{i+1}u$ and 
$er = \frac{10eu}{\sqrt n} < u^{1-0.1\delta}$}\\
&= \left(\frac{10e\, n^{1/2-\delta} r}
{u^{1-0.1\delta}}\right)^{(1-\delta)^{i+1} u}\cdot u! \\
&\leq \left(\frac{200e\, u^{0.1\delta}}
{n^{\delta}}\right)^{(1-\delta)^{i+1} u}\cdot u! 
\tag{since $r=\frac{10u}{\sqrt n}$}\\
&< \frac{2^{i+1}}{n^M} u!.
\tag{since $u<n$ and $n$ is sufficiently large}
\end{align*}
provided that $n$ is sufficiently large.%
\footnote{
While this last inequality might seem lossy
since the expression is actually less than $\frac 1{n^M}u!$,
the step that determines the factor $2^{i+1}$
is step (2), in which we have to meet 
the required condition to apply density boost.}
Thus, we have $|\calF_{N+1}| <  \frac{2^{i+1}}{n^M}\cdot u!$, 
as desired. 
\item[(iii)]
To show (iii), assume for the sake of contradiction 
that $|S_i\cap S_j| < t'$ for some $i,j$. 
From Step (3), $\calF_i(S_i)$ and $\calF_j(S_j)$
are both $r$-spread. 
Let $x=t-|S_i\cap S_j|$
and define 
\begin{align*}
\calG_i &= \left\{X\in\calF_i(S_i) : |X\cap (S_j\setminus S_i)| \leq 
\left\lfloor \tfrac {x-1}2\right\rfloor \right\}
\quad\text{and} \\
\calG_j &= \left\{X\in\calF_j(S_j) : |X\cap (S_i\setminus S_j)| \leq 
\left\lfloor \tfrac {x-1}2\right\rfloor\right\}.
\end{align*}
We bound the size of $\calG_i$. 
Note that an element of $\calF_i(S_i)\setminus\calG_i$
must contain a subset of size $\left\lfloor \tfrac {x-1}2\right\rfloor
+ 1 = \lceil x/2\rceil$ of $S_j\setminus S_i$.
By taking union bound across all such subsets
and using spreadness of $\calF_i(S_i)$, we get that
\begin{align*}
|\calF_i(S_i)\setminus\calG_i| &\leq 
\binom{|S_j\setminus S_i|}{\lceil x/2\rceil } r^{-\lceil x/2\rceil} 
|\calF_i(S_i)| \\
&\leq \binom{x+u}{\lceil x/2\rceil} 
r^{-\lceil x/2\rceil} |\calF_i(S_i)|
\tag{since $|S_j\setminus S_i| \leq n-|S_i\cap S_j| 
= u+x$} \\
&\leq \left(\frac{e(x+u)}{\lceil x/2\rceil}\right)^{\lceil x/2\rceil}  
r^{-\lceil x/2\rceil} 
|\calF_i(S_i)| \tag{using \Cref{lem:binom_bound}} \\
&\leq \left(\frac 1r\left(2e+\frac{2eu}x\right)\right)
^{\lceil x/2\rceil}
|\calF_i(S_i)| \\ 
&\leq\ \tfrac 12 |\calF_i(S_i)|
\end{align*}
where the last inequality follows from 
$x\geq t-t' \geq n^{1/2} = 10\frac ur$, and we can make $n$ 
(hence $r$) sufficiently large.

Thus, $|\calG_i|\geq \tfrac 12|\calF_i(S_i)|$,
We note that this implies that $\calG_i$ is $(r/2)$-spread. 
Indeed, for any nonempty $X$, we have
$$|\calG_i[X]| \leq |\calF_i(S_i)[X]| 
\leq r^{-|X|} |\calF_i(S_i)| 
\leq 2r^{-|X|} |\calG_i| \leq (r/2)^{-|X|} |\calG_i|.$$
Similarly, $\calG_j$ is $(r/2)$-spread. 
Now, by \Cref{cor:spread_lemma},
(since $r/2 = \frac{5u}{\sqrt n} >
5 n^{2\delta} > 2^{12}\log_2(2n)$ for all sufficiently large $n$),
there exists $G_i\in\calG_i$ 
and $G_j\in\calG_j$ such that $G_i\cap G_j=\emptyset$.
Note that $G_i\cup S_i\in \calF_i$ 
and $G_j\cup S_j\in\calF_j$ and 
\begin{align*}
|(G_i\cup S_i) \cap (G_j\cup S_j)| 
&\leq |S_i\cap S_j| + |G_i\cap S_j| + |G_j\cap S_i|  \\
&= |S_i\cap S_j| + |G_i\cap (S_j\setminus S_i)| 
+ |G_j\cap (S_i\setminus S_j)|  \\
&\leq (t-x) + \tfrac{x-1}2 + \tfrac{x-1}2 < t,
\end{align*}
contradicting the $t$-intersecting property of $\calF$.
This gives that $\calS$ is $t'$-intersecting.
\end{enumerate}
This completes the proof of the assertion $A_{i+1}$,
completing the induction.
\end{proof}
By using $B_i$ of the above lemma
where 
\begin{itemize}
\item $i$ is the largest integer such that 
$(1-\delta)^i u > n^{1/2+2\delta}$,
so $(1-\delta)^{i+1} u < n^{1/2+2\delta}$.
(so $i\leq C\log n$ for some constant $C=C(\delta)$)
and
\item $M > 2C\log 2+1$, so $2^{i+1} \leq n^{M-1}$,
\end{itemize}
we get that there exists $X\subseteq [n]^2$
such that
$$|\calF(X)| \geq 
\frac 1{(10n^{1/2-\delta})^{(1-\delta)^{i+1} u}}
|\calF|
\geq \frac 1{(10n^{1/2-\delta})^{n^{1/2+2\delta}}}
\geq \frac 1{n^{n^{1/2+2\delta}}} |\calF|.$$
\begin{corollary}[Strong Density Boost]
\label{cor:strong_density_boost}
Let $\delta \in \left(0,\tfrac 12\right)$. 
There exists $n_0=n_0(\delta)$
such that for every $n>n_0$, $n^{0.99} <t < n-n^{1/2+2\delta}$,
and $t$-intersecting family $\calF\subseteq\Sigma_n$ 
with $|\calF| \geq \frac{1}{n}\cdot u!$,
there exists $X\subseteq [n]^2$
such that $|X| = t-n^{1/2}$ and 
\begin{equation}
\label{eq:density_boost}
|\calF(X)|
\geq \frac 1{n^{n^{1/2+2\delta}}}|\calF|.
\end{equation}
\end{corollary}
\subsection{Accurate Spread Approximation}
\label{subsec:accurate}
In this section, 
we complete the proof of \Cref{thm:accurate_spread_approx}.
First, we extract a consequence of \Cref{cor:strong_density_boost}.
\begin{corollary}
\label{cor:spreadness}
Let $\eps\in \left(0,\tfrac 12\right)$.
Then there exists $n_0=n_0(\eps)$ such that 
for every $n>n_0$, $n^{0.99} < t < n-n^{1/2+\eps}$,
and $t$-intersecting family $\calF\subseteq\Sigma_n$ 
with $|\calF| \geq \frac 1n\cdot u!$,
there exists $Y\subseteq [n]^2$ such that 
$|Y| < t+0.001\,u$ and $\calF(Y)$ is $0.1u$-spread.
\end{corollary}
\begin{proof}
Take $X$ as in \Cref{cor:strong_density_boost}
with $\delta = \frac{\eps}4$,
so $|X| = t-n^{1/2 }$.
We again use \Cref{obs:finding_spread},
take maximal $Y\supseteq X$ such that 
\begin{equation}
\label{eq:spreadness_2}
    |\calF(Y)| \geq (0.1u)^{-|Y|+|X|}|\calF(X)|,
\end{equation}
so $|\calF(Y)|$ is $0.1u$-spread.
Now, we bound the size: assume that $|Y| > t+0.001u$.
Let $d = |Y|-t$.
\begin{align*}
(u-d)! = (n-|Y|)! &\geq |\Sigma_n(Y)| \\
&\geq 
|\calF(Y)|  \\
&\geq (0.1u)^{-(d+n^{1/2})}\, |\calF(X)| 
\tag{from \eqref{eq:spreadness_2}}\\
&\geq (0.1u)^{-(d+n^{1/2})}\, n^{-n^{(1+\eps)/2}}|\calF| 
\tag{from \eqref{eq:density_boost}}\\
&\geq \frac 1n (0.1u)^{-(d+n^{1/2})}\, n^{-n^{(1+\eps)/2}}
\cdot u! \tag{since $|\calF|\geq \frac 1n u!$}\\
&\geq \frac 1n (0.1u)^{-d}\,n^{-n^{1/2}}\, n^{-n^{(1+\eps)/2}}
\cdot u! \tag{since $u\leq n$} \\
&\geq (0.1u)^{-d}\, n^{-2n^{(1+\eps)/2}}\cdot u!
\tag{since $n^{1/2} < n^{(1+\eps)/2}$}
\end{align*}
Thus, we deduce that (using \Cref{lem:factorial_bound})
\begin{align*}
(0.1u)^d n^{2n^{(1+\eps)/2}}
&\geq \frac{u!}{(u-d)!} 
\geq \left(\frac ue\right)^d,
\end{align*}
which implies that 
$$n^{2n^{(1+\eps)/2}+1} \geq 
\left(\frac{10}e\right)^d
\geq \left(\frac{10}e\right)^{0.001u},$$
which is a contradiction for all sufficiently 
large $n$ because $2n^{(1+\eps)/2}\log n \ll n^{1/2+\eps} < u$.
\end{proof}

We are now ready to prove \Cref{thm:accurate_spread_approx}.
The proof uses spread approximation similar to 
proving the statements $A_i$ in \Cref{lem:iterative}.
For reader's convenience, we reproduce the theorem statement below.
\spreadapprox*
\begin{proof}[Proof of \Cref{thm:accurate_spread_approx}]
Let $q=t+0.001u$.
Initiate $\calF_1=\calF$. For each $j\geq 1$, do the following.
\begin{enumerate}
\item If $|\calF_j| < \frac 1n\cdot u!$, then terminate.
\item Take a maximal (under inclusion) $S_j\subseteq [n]^2$
such that $|S_j| < q$
and $\calF_j(S_j)$ is $0.1u$-spread.
Such a set exists by \Cref{cor:spreadness}.
\item Let $\calF_{j+1} = \calF_j\setminus\calF_j[S_j]$.
\end{enumerate}
The algorithm terminates when $|\calF_j| < \frac 1n \cdot u!$.
Suppose this happens when $j=N+1$.
We take $\calS=\{S_1,\dots,S_N\}$. 
We claim that this works. Note that 
(a) and (b) are clearly satisfied 
(for (b), note that $\calF\setminus \calF[\calS]\subseteq \calF_{N+1}$).
We now verify (c).
Assume for the sake of contradiction 
that $|S_i\cap S_j| < t$ for some $i,j$. 
From Step (3), $\calF_i(S_i)$ and $\calF_j(S_j)$
are both $0.1u$-spread. 
Let $x=t-|S_i\cap S_j|$ and define 
\begin{align*}
\calG_i &= \left\{X\in\calF_i(S_i) : |X\cap (S_j\setminus S_i)| \leq 
\left\lfloor \tfrac{x-1}2\right\rfloor\right\}
\quad{\text{and}} \\
\calG_j &= \left\{X\in\calF_j(S_j) : |X\cap (S_i\setminus S_j)| 
\leq\left\lfloor \tfrac{x-1}2\right\rfloor\right\}
\end{align*}
We bound the size of $\calG_i$. Note that an element of $\calF_i(S_i)\setminus\calG_i$
must contain a subset of size 
$\left\lfloor \tfrac{x-1}2\right\rfloor 
+ 1 = \left\lceil \tfrac x2\right\rceil$ 
of $S_j\setminus S_i$.
By taking union bound across all such subsets
and using spreadness of $\calF_i(S_i)$, we get that
\begin{align*}
|\calF_i(S_i)\setminus\calG_i| &\leq 
\binom{|S_j\setminus S_i|}{\lceil x/2\rceil}
 (0.1u)^{-\lceil x/2\rceil} |\calF_i(S_i)| \\
&\leq \binom{x+0.001u}{\lceil x/2\rceil} 
\tag{since $|S_j\setminus S_i| = |S_j|-|S_i\cap S_j| 
\leq (t+0.001u)-(t-x)$}
(0.1u)^{-\lceil x/2\rceil} |\calF_i(S_i)| \\
&\leq \left(\frac{e(x+0.001u)}{\lceil x/2\rceil}
\right)^{\lceil x/2\rceil}
(0.1u)^{-\lceil x/2\rceil} |\calF_i(S_i)| 
\tag{using \Cref{lem:binom_bound}}\\
&\leq \left(\frac{e(x+0.001u)}{0.1\cdot u\cdot\frac x2}
\right)^{\lceil x/2\rceil} |\calF_i(S_i)| \\
&\leq \left(\frac {20e}{u}+\frac {0.02e}x\right)
^{\lceil x/2\rceil}
|\calF_i(S_i)| \\ 
&\leq \tfrac 12 |\calF_i(S_i)|,
\end{align*}
where the last inequality follows from 
$\frac{20e}u < 0.1$ and $\frac{0.02e}x < 0.1$ (since $x\geq 1$).

Thus, $|\calG_i|\geq \tfrac 12|\calF_i(S_i)|$,
and so $\calG_i$ is $0.05u$-spread. 
Similarly, $\calG_j$ is $0.05u$-spread. 
Now, by \Cref{cor:spread_lemma},
(since $0.05u>2^{12}\log_2(2n)$ for all sufficiently large $n$),
there exists $G_i\in\calG_i$ 
and $G_j\in\calG_j$ such that $G_i\cap G_j=\emptyset$.
Note that $G_i\cup S_i\in \calF_i$ 
and $G_j\cup S_j\in\calF_j$ and 
\begin{align*}
|(G_i\cup S_i) \cap (G_j\cup S_j)| 
&\leq |S_i\cap S_j| + |G_i\cap S_j| + |G_j\cap S_i|  \\
&= |S_i\cap S_j| + |G_i\cap (S_j\setminus S_i)| 
+ |G_j\cap (S_i\setminus S_j)|  \\
&\leq (t-x) + \tfrac{x-1}2 + \tfrac{x-1}2 < t,
\end{align*}
contradicting the $t$-intersecting property of $\calF$.
Thus, $\calS$ is $t$-intersecting.
\end{proof}
\section{Peeling and Rough Bound on
\texorpdfstring{$|\calW_k|$}{|W\_k|}}
\label{sec:peeling}
In this and the next section, we will work exclusively 
with $\calS$ found by \Cref{thm:accurate_spread_approx}.
In particular, let $\calS\subseteq \binom{[n]^2}{\leq q}$ 
be a $t$-intersecting family (so $q$ would be $t+0.001u$
if $\calS$ comes from applying \Cref{thm:accurate_spread_approx}).

To analyze the structure of $\calS$ further,
we introduce the construction known as \vocab{peeling},
which was first introduced by Kupavskii and Zakharov in \cite{spread_approx},
refined in \cite{partitions}, and used more extensively in 
\cite{one_minus_eps}.
Most material in this section comes from 
\cite[\S 2.4]{one_minus_eps}, but we will need more refined bounds 
than that paper.
We will give all proofs, making our paper self-contained.

Roughly speaking, we alternate between removing 
partial permutations with the highest number of elements 
from $\calS$ and simplifying $\calS$ 
(while keeping it still $t$-intersecting).

\subsection{Peeling: Definitions and Basic Properties}
\label{subsec:peeling}
\begin{definition}
Given a $t$-intersecting family $\calS\subseteq\binom{[n]^2}{\leq q}$,
its \vocab{simplification} $\calS'$ is obtained by 
repeating the following operations until no further application 
of either operation is possible.
\begin{itemize}
\item Remove any element $T_1\in\calS$ that is contained 
in another element $T_2\in\calS$,
if any such $T_1$ exists.
\item For each $S\in\calS$ and $X\subsetneq S$,
replace $S$ with $X$ if it would still result in a 
$t$-intersecting family.
\end{itemize}
The simplification $\calS'$ has the following properties:
\begin{enumerate}[label=(\alph*)]
\item $\calU[\calS] \subseteq \calU[\calS']$
for any $\calU\subseteq 2^{[n]^2}$.
\item There are no $T_1,T_2\in\calS'$ such that 
$T_1\subsetneq T_2$.
\item \label{item:maximality}
For any $S\in\calS'$ and $X\subsetneq S$,
there exists $T\in\calS'$ such that $|X\cap T| \leq t-1$.
\end{enumerate}
\end{definition}

We iteratively define the following families,
which will be used for the entire paper:
\begin{itemize}
\item Let $\mathvocab{\calT_{q-t}}$ be the simplification of $\calS$.
\item For each $k=q-t,q-t-1,\dots,1$, we let 
$\mathvocab{\calW_k} 
:= \calT_k\cap\binom{[n]^2}{t+k}$
and $\mathvocab{\calT_{k-1}}$ be the simplification of $\calT_k\setminus \calW_k$.
\end{itemize}

The basic properties of peeling are summarized 
in the following proposition.
Part (c) and (d) in particular are the most crucial properties 
that we will use repeatedly.
\begin{proposition}[{\cite[Lem.~9]{one_minus_eps}, 
    \cite[Lem.~16]{partitions}}] \hspace{0pt}
\label{prop:peeling_properties}
For any positive integer $k\leq q-t$, all of the following hold.
\begin{enumerate}[label=(\alph*)]
\item All sets in $\calW_k$ have size exactly $t+k$.
All sets in $\calT_k$ have size at most $t+k$.
\item For all $\calF\subseteq 2^{[n]^2}$, we have 
$\calF[\calT_k] = \calF[\calT_{k-1}]\cup \calF[\calW_k]$.
\item There does not exist a subset $X\subseteq [n]^2$
and a subfamily $\calY\subseteq \calT_k(X)$ 
such that $|\calY|>1$ and $\calY$ is
$\alpha$-spread for some $\alpha>k$.
\item For all $X\subseteq [n]^2$, we have 
$|\calW_k(X)| \leq k^{t+k-|X|}$.
\end{enumerate}
\end{proposition}
\begin{proof}
Part (a) follows directly from the definition of 
$\calW_k$ and $\calT_k$.
To prove (b), note that by property (a), simplification,
we have that 
$$\calF[\calT_k] = \calF[\calT_k\setminus \calW_k] 
\cup \calF[\calW_k] = \calF[\calT_{k-1}]\cup \calF[W_k]$$
We now prove (c). Assume that there exists such a subset $X$.
If $X$ is not contained in some $S\in\calT_k$, 
then $\calT_k(X)$ is empty, so we are done.
Thus, $X\subseteq S$ for some $S\in\calT_k$.
If $X=S$, then $|\calY|\leq |\calT_k(X)| = 1$,
which is a contradiction as well.
Thus, $X\subsetneq S$, so
by the property \ref{item:maximality} of simplification, 
there exists $S\in\calT_k$ 
such that $|X\cap S| = z < t$. 
For any $R\in\calY$, we have $|R\cap S|\geq t$,
which implies that $|R\cap (S\setminus X)|\geq t-z$.
Therefore, $R$ contains a subset in $\binom{S\setminus X}{t-z}$,
which implies that 
\begin{align*}
    |\calY[X]| 
    &\leq \sum_{U\in \binom{S\setminus X}{t-z}} 
    |\calY[X\cup U]| \\
    &\leq \binom{|S\setminus X|}{t-z}
    \alpha^{-(t-z)} |\calY[X]| 
    \tag{$\alpha$-spreadness of $\calY(X)$} \\
    &\leq \binom{t+k-z-1}{t-z} \alpha^{-(t-z)} |\calY[X]|.
\end{align*}
However, observe that 
$$\binom{t+k-z-1}{t-z} = \frac{k}{1}\cdot  \frac{k+1}2\cdot\frac{k+2}3 
\cdots \frac{k+t-z-1}{t-z} < k^{t-z},$$
which gives a contradiction.

Finally, to prove (d), we assume that $|\calW_k(X)| = \alpha^{t+k-|X|}$,
where $\alpha>k$.
Utilizing \Cref{obs:finding_spread},
we select maximal set $Y\supseteq X$ such that 
$|\calW_k(Y)|\geq  \alpha^{|X|-|Y|}  |\calW_k(X)|$.
Then $\calW_k(Y)$ is $\alpha$-spread.
Furthermore,
$$|\calW_k(Y)| \geq \alpha^{|X|-|Y|}|\calW_k(X)| 
> \alpha^{t+k-|Y|} \geq 1,$$
implying that $|\calW_k(Y)| > 1$.
Thus, $\calY=\calW_k(Y)$ violates (c).
\end{proof}

\subsection{Rough Bound on \texorpdfstring{\boldmath$|\calW_k|$}{|W\_k|}}
We now use \Cref{prop:peeling_properties} 
to give a rough bound on $|\calW_k|$ that works for any $k$.
The bound in this section will be used for 
larger $k$ (more precisely, $k>t^{2/7-\eps/2}$).
We call this rough because in \Cref{sec:bounding},
we will prove a stronger estimate that works 
for smaller $k$.

\begin{proposition}
\label{prop:bound_W_k}
For any positive integer $k\leq q-t$, we have 
\begin{equation}
\label{eq:bound_W_k}    
|\calW_k| \leq \sum_{j=0}^k \binom tj \binom kj^2 k^{k-j}.
\end{equation}
\end{proposition}
\begin{proof}
We pick two sets $A,B\in\calT_k$ such that $|A\cap B|=t$,
which must exist since $\calT_k$ is simplified 
(if $|A\cap B|\geq t+1$ for all $A,B\in\calT_k$, 
then removing an element $x\in [n]^2$
from every set $T\in \calT_k$
will still make $\calT_k$ $t$-intersecting,
contradicting the fact that $\calT_k$ is simplified).

Let $I = A\cap B$. For any $C\in\calW_k$,
let $C_0=C\cap I$, $C_1=C\cap (A\setminus I)$, and
$C_2=C\cap (B\setminus I)$. 
See \Cref{fig:venn_diagram} for the diagram.

\begin{figure}[ht]
\begin{asy}
size(4cm,0);
settings.outformat="pdf";
defaultpen(fontsize(10pt));
draw(circle(dir(90),sqrt(3)));
filldraw(circle(dir(210),sqrt(3)), black+opacity(0.1), linewidth(1));
filldraw(circle(dir(330),sqrt(3)), black+opacity(0.1), linewidth(1));
label("$A$", (1+sqrt(3))*dir(210), dir(210));
label("$B$", (1+sqrt(3))*dir(330), dir(330));
label("$C$", (1+sqrt(3))*dir(90), dir(90));
label("$C_0$", (0,0));
label("$C_1$", 1.2*dir(150));
label("$C_2$", 1.2*dir(30));
\end{asy}
\caption{Venn diagram in the proof of \Cref{prop:bound_W_k}}
\label{fig:venn_diagram}
\end{figure}
If $|C_0| = t-j$, then $|C_1|\geq j$ and $|C_2|\geq j$
because $\calW_k$ is $t$-intersecting.
Let $C_1'$ and $C_2'$ be arbitrary $j$-element 
subsets of $C_1$ and $C_2$, respectively.
There are $\binom t{t-j} = \binom tj$ choices for $C_0$,
at most $\binom{|A|-t}{j}$ choices for $C_1'$,
and at most $\binom{|B|-t}{j}$ choices for $C_2'$.
Furthermore, the remaining component 
$X = C\setminus (C_0\cup C_1'\cup C_2')$
must be picked from $\calW_k(C_0\cup C_1'\cup C_2')$,
which by \Cref{prop:peeling_properties} (d),
has cardinality at most $k^{k-j}$.
Thus, we indeed have 
\[|\calW_k| \leq \sum_{j=0}^k \binom{t}{j} 
\binom{|A|-t}{j} \binom{|B|-t}j k^{k-j}
\leq \sum_{j=0}^k \binom tj \binom kj^2 k^{k-j}.\qedhere \]
\end{proof}
We estimate the right hand side of \eqref{eq:bound_W_k}
by looking for the term that gives the maximum value.
In particular, we find $j$ such that $f(j) := 
\binom tj \binom kj^2 k^{k-j}$ is maximum.
A straightforward computation gives 
\begin{equation}
\label{eq:solve_j}
\frac{f(j)}{f(j+1)} 
= \frac{\binom tj}{\binom t{j+1}}\cdot
\frac{\binom kj^2}{\binom k{j+1}^2}\cdot
k
= \frac{(j+1)^3 k}{(t-j)(k-j)^2},
\end{equation}
which is an increasing function. 
We check whether $\tfrac{f(j)}{f(j+1)}$ is greater 
or less than $1$ or not. 
The calculation results in the following claim.
\begin{proposition}
\label{prop:max_j}
Let $k\leq q-t$ be a nonnegative integer.
Let $j_0$ be the index $j\in [0,\min(t,k)]$ such that 
$f(j)=\binom tj \binom kj^2 k^{k-j}$ 
attains the maximum.
Assume $t>n^{0.99}$ and 
$k>k_0$ for some constant $k_0$.
\begin{enumerate}[label=(\alph*)]
\item If $k=o(t^{1/4})$, then $j_0=k$.
\item If $k\leq t^{1/2}$, then $j_0\geq \frac 12 k$.
\item If $k\geq t^{1/2}$, then $j_0 \geq \frac 12(tk)^{1/3}$.
\end{enumerate}
\end{proposition}
\begin{proof}
\begin{enumerate}[label=(\alph*)]
\item  If $k=o(t^{1/4})$, then 
$$\frac{f(k-1)}{f(k)} = \frac{k^4}{t-k+1} \ll 1,$$
so we have that $f$ is increasing,
implying that it is maximized at $k$.
\item Assume that $j_0\leq \frac k2$. Then 
$t-j_0 \geq 0.99k^2$ and $k-j_0 \geq \frac k2$, so 
$$\frac{f(j_0)}{f(j_0+1)} \leq 
\frac{\left(k/2+1\right)^3\cdot  k}{0.99 k^2 
\cdot \left(k/2\right)^2} < 0.6,$$
if $k$ is sufficiently large.
This contradicts $f(j_0) \geq f(j_0+1)$.
\item Assume that $j_0 \leq \frac 12 (tk)^{1/3}$.
Then since $k\geq t^{1/2}$, we have 
$t\leq k^2$, so
$j_0 \leq \frac 12 k$. In particular, $k-j_0\geq \frac 12k$.
Moreover, since $k\leq q-t\leq n\ll t^2$
(this uses $t>n^{0.99}$), we have 
$j_0\leq \frac 12(tk)^{1/3} 
< 0.01t$, so $t-j_0\geq 0.99t$. Using these two inequalities gives
$$\frac{f(j_0)}{f(j_0+1)} \leq 
\frac{\left(\frac 18 tk\right) \cdot k}{0.99 t \cdot (\frac 12 k)^2}
< 0.9$$
if $k$ is sufficiently large.
This contradicts $f(j_0)\geq f(j_0+1)$.
\qedhere
\end{enumerate}
\end{proof}
\begin{corollary}
\label{cor:bound_W_k}
Assume that $t>n^{0.99}$ and $k>k_0$ for some constant $k_0$. 
Then we have
$$|\calW_k| \leq \left(200\max\left(k,\frac tk\right)\right)^k.$$
\end{corollary}
\begin{proof}
Note that $|\calW_k| \leq (k+1)f(j_0)$.
Using \Cref{lem:binom_bound},
we get that 
\begin{equation}
\label{eq:W_k_j_0}
|\calW_k| \leq (k+1)\binom t{j_0} \binom k{j_0}^2 k^{k-j_0} 
\leq (k+1)\left(\frac{e^3 tk^2}{j_0^3}\right)^{j_0} k^{k-j_0}.
\end{equation}
Now, we consider two cases.
\begin{itemize}
\item If $k\leq t^{1/2}$, then we apply 
\Cref{prop:max_j} (b) to get that $j_0\geq \frac k2$.
Plugging this in gives
\begin{align*}
|\calW_k| &\leq (k+1) \left(\frac{e^3 tk^2}{(k/2)^3}\right)^{j_0} 
k^{k-j_0} \\
&\leq (k+1)  \left(\frac{199 t}k\right)^{j_0}  
\cdot k^{k-j_0} 
\tag{since $8e^3 < 199$}\\
&\leq (k+1)  \left(\frac{199 t}k\right)^{j_0}  
\cdot \left(\frac tk\right)^{k-j_0} 
\tag{since $k^2\leq t$, so $k\leq\tfrac tk$}\\
&\leq (k+1) \left(\frac{199 t}k\right)^{k}  
\leq \left(\frac{200t}k\right)^k.
\end{align*}
\item If $k\geq t^{1/2}$, then we apply \Cref{prop:max_j} (c)
to get that $j_0^3 \geq \frac 18tk$.
Plugging this in gives 
$$|\calW_k| \leq (k+1) \left(\frac{8e^3 tk^2}{kt}\right)^{j_0}
k^{k-j_0} \leq (k+1) \cdot (199k)^{j_0} k^{k-j_0} 
\leq (200k)^k$$
(in the middle inequality, we use $8e^3 < 199$).
\qedhere
\end{itemize}
\end{proof}
In the main proof in \Cref{sec:conclusion},
we will apply \Cref{cor:bound_W_k} to bound the 
size of $\calW_k$ (and hence the size of 
$\calF[\calW_k]$) for each $k>t^{2/7-\eps/2}$.
In particular, this range of $k$ gives 
the bound of very roughly 
$|\calW_k| \ll t^{5k/7} \approx (n-t)^k$
(in the regine $n-t\approx t^{5/7}$),
so we get (again, very roughly) that 
$|\calF[\calW_k]| \ll (n-t-k)! (n-t)^k 
\approx (n-t)!$. 

However, when $k$ is smaller than $t^{2/7-\eps/2}$,
the bound above no longer works.
The main job of \Cref{sec:bounding} is to 
handle this range of $k$ by choosing a specific 
layer to stop peeling and analyzing that layer 
more carefully.

\subsection{Extended Remark on Previous Proofs}
We close this section by briefly explaining what 
Kupavskii's proof does in \cite{one_minus_eps}
and explain the exact range of $n$ and $t$ it works.
\begin{remark}[Spread approximation step in \cite{one_minus_eps}]
\label{rem:loglog}
The spread approximation step in \cite{one_minus_eps} 
works as follows.
\begin{itemize}
\item First, a coarser spread approximation was obtained.
The resulting statement is similar to the statement $A_i$ in
\Cref{lem:iterative},
but the range of $n$ and $t$ is more limited,
and the resulting family is only 
$t'$-intersecting for some $t'$ slightly 
smaller than $t$.
\item Then, the density boost result was proved 
using the peeling argument and \Cref{prop:bound_W_k}.
\item Finally, a more accurate spread approximation 
(same statement as \Cref{thm:accurate_spread_approx},
but in a more limited range of $n$ and $t$) was 
obtained.
\end{itemize}
We explain why the first and second step only work when 
$t\leq n - \frac{Cn\log \log n}{\log n}$.
First, we note that our bound of $|\calW_k|$ in \Cref{cor:bound_W_k}
is essentially the same as their bound of $|\calW_k|$
in \cite[Lem.~11]{one_minus_eps}.
The proof of this requires $k\leq \frac{n-t}{50} = \frac{u}{50}$. 
In particular, it requires the spread approximation 
to take the uniformity (i.e., the maximum size 
of each eleemnt of $\calS$) down to $t+\frac{u}{50}$
because our bound has no control of higher uniformities 
than that.

Now, suppose that we want to find a spread approximation 
with uniformity $q=t+\frac{u}{50}$. 
We would have to run the usual argument as follows:
initiate $\calF_1=\calF$. For each $i\geq 1$, do the following.
\begin{itemize}
\item Take maximal (under inclusion) $S_i$ such that 
$|\calF_i[S_i]| \geq r^{-|S_i|}|\calF_i|$.
\item If $|S_i| > q$, then terminate.
\item Otherwise, let $\calF_{i+1} = \calF_i\setminus\calF_i[\calS_i]$.
\end{itemize}
Now suppose that the algorithm terminates after $N$ steps.
Then we have
$$|\calF\setminus\calF[\calS]| 
\leq |\calF_N| \leq r^q |\calF_N[S_N]| 
\leq r^q (n-q)! \leq r^q (0.98u)!.$$
Wanting this to be less than $u!$, we want (roughly)
$r^q \leq u^{0.02u}$.
However, we also have a critical condition 
in the spread lemma (\Cref{cor:spread_lemma}) that 
$r$ must be at least $2^{12}\log_2(2n)$.
In particular, we must have $0.02 u\log u \geq q\log r 
\gtrsim n\log\log n$,
so we need $u\gtrsim \frac{n\log\log n}{\log n}$
for the argument to work.
This explains why Kupavskii's proof works only up to 
$t\leq n-\frac{C n\log\log n}{\log n}$.
In contrast, our argument in \Cref{sec:spread_approx}
does not rely on peeling and gives a better result.
\end{remark}
\begin{remark}[Peeling step in \cite{one_minus_eps}]
\label{rem:34}
After \Cref{cor:bound_W_k}, Kupavskii's proof proceeds to 
select $k$ such that $|\calW_k|$ is large enough
and use \Cref{prop:bound_W_k} to bound the size $|\calW_k|$.
In particular, Kupavskii needs to show 
$|\calW_k| \leq (1+o(1))\binom tk$.
However, upon considering \Cref{prop:bound_W_k},
one gets that $\sum_{j=0}^k \binom tj \binom kj^2 k^{k-j}$
is close to by $\binom tk$ (which is the term when $j=0$)
only if $j_0=0$ and $f(0)\gg f(1)$; the later happens 
only if $k\ll n^{3/4}$, or $n-t\gg n^{3/4}$
(since $k$ is supposed to be approximately $\frac n{n-t}$
in the equality case).
This explains why Kupavskii's proof 
works only up to exponent $\tfrac 34$,
and getting to $\tfrac 57$ requires a new argument,
which we will introduce in the next section.
\end{remark}

\section{Upper Bound on \texorpdfstring{$|\calT_k|$}{|T\_k|}}
\label{sec:bounding}
In this section, we continue analyzing $\calS$.
In particular, let $\calS\subseteq \binom{[n]^2}{\leq q}$ 
be a $t$-intersecting family, and
let $\calT_k$ and $\calW_k$ be as defined in 
\Cref{subsec:peeling}.
We commit to a choice of $k$ that will show that 
$\calF$ is close to $\calA_k$:
\begin{equation}
\label{eq:cond_k}
\boxed{\text{Let } \mathvocab k \text{ be the largest nonnegative integer 
such that } k\leq t^{2/7- \eps/2} \text{ and } 
|\calW_k| > t^{-1/7}\tbinom tk.}
\end{equation}

For this section, we assume that such a $k$ exists.
(If $k$ does not exist, then the bounds in \Cref{sec:peeling}
suffices.)
We will now use this choice of $k$ to get a sharper bound on 
the size of $|\calT_k|$.
The final result that we will use is \Cref{cor:bound_T_k}.

\subsection{Existence of Good Sets}
The idea of this section is to find sets 
$A_1,A_2,\dots,A_r$ such that $|A_1\cap \dots\cap A_r|$
is as small as possible.
We pick an integer $\mathvocab r$ sufficiently large such that 
$r\geq 100$ and
\begin{equation}
\label{eq:cond_r}
\frac{3r-1}{r-1}\cdot \left(\frac 27 - \frac\eps 2\right)
< \frac 67 - \eps.
\end{equation}
This $r$ exists because as $r\to\infty$, the left 
side converges to $3\cdot \left(\frac 27 - \frac\eps 2\right)
= \frac 67 - \frac 32\eps$,
which is less than $\frac 67-\eps$, so such $r$ exists.
Note that $r$ depends only on $\eps$,
so we will treat it as an absolute constant.
\begin{lemma}
\label{lem:good_A}
Given the above setup, there exist $A_1,\dots,A_r\in\calW_k$ such that 
$$|A_1\cap \dots \cap A_r| = t-(r-2)k.$$
\end{lemma}
Before proving the lemma, we discuss some of its consequences.
Note that for any $A_1,\dots,A_r\in\calW_k$, we always have 
$|A_1\cap\dots\cap A_r| \geq t-(r-2)k$.
This is because for any $1\leq i\leq n$, we have
\begin{align*}
|A_1\cap\dots\cap A_r| &\geq |A_i| - |A_i\setminus A_2| 
- \dots - |A_i\setminus A_r| \\
&\geq (t+k) - (r-1)\cdot k \tag{since $|A_i\cap A_j|\geq t$
for all $i$ and $j$} \\
&= t-(r-2)k,
\end{align*}
so the value $t-(r-2)k$ is as small as possible.
Moreover, if $|A_1\cap \dots\cap A_r|=t-(r-2)k$,
then the equalities above hold for all $i$. In particular,
\begin{itemize}
\item the first inequality is an equality for all $i$ only if 
$|A_i\cap A_j|=t$ for all $1\leq i\neq j\leq r$.
\item the second inequality is an equality for all $i$ only if 
$A_i\setminus A_j$ and $A_i\setminus A_\ell$ 
are disjoint for all $i$ and $j\neq \ell$.
In particular, any element in $A_1\cup \dots\cup A_r$
must appear in at least $r-1$ of the sets 
$A_1,\dots,A_r$.
\end{itemize}
Thus, we deduce that for $A_1,\dots,A_r$ in the lemma,
we have $|A_i\cap A_j|=t$ for all $1\leq i\neq j\neq r$
and every element in $A_1\cup\dots\cup A_r$
must appear in at least $r-1$ of the sets 
$A_1,\dots,A_r$.

We now prove this lemma.
\begin{proof}[Proof of \Cref{lem:good_A}]
Across all possible tuples 
$(A_1,\dots,A_r)$ of element in $\calW_k$, we take one 
such that $|A_1\cap\dots\cap A_r|$ is as small as possible.
Assume for contradiction that $|A_1\cap \dots\cap A_r| 
= t-(r-2)k+\mathvocab x$ for some $x>0$.
We let 
$$\mathvocab{a_i} := \text{the number of elements in } 
[n]^2 \text{ that are in exactly }i 
\text{ sets among } A_1,\dots,A_r.$$
In particular, $a_r=|A_1\cap \dots\cap A_r| = 
t-(r-2)k+x$.
We make several observations that we will use.
\begin{itemize}
\item Since $|A_j|=t+k$, summing across all $j$ gives
\begin{equation}
\label{eq:bound_a_1}
a_1 + 2a_2  +\dots + ra_r = r(t+k).
\end{equation}
\item Since $|A_j\cap A_\ell|\geq t$,
summing this across all $1\leq j\neq \ell\leq r$ gives
\begin{equation}
\label{eq:bound_a_2}
\sum_{i=1}^{r-1} i(i-1)a_i \geq r(r-1)t
\end{equation}
(since for each element $x$ counted in $a_i$,
there are $i(i-1)$ choices of $(j,\ell)$ such that 
$x$ is in both $A_j$ and $A_\ell$).
\item We give an upper bound on $a_{r-1}$.
To do this, we use \eqref{eq:bound_a_1} to get 
that $(r-1)a_{r-1}+ra_r \leq r(t+k)$, 
which rearranges to
\begin{equation}
\label{eq:bound_A_r-1}
a_{r-1} \leq \frac{r(t+k)-ra_r}{r-1}
= \frac{r(t+k) - r(t-(r-2)k+x)}{r-1} 
= rk-\frac{rx}{r-1}.
\end{equation}
\item Next, we give an upper bound on $a_1,a_2,\dots,a_{r-2}$.
To do this, we subtract \eqref{eq:bound_a_2} 
from $(r-2)$ times \eqref{eq:bound_a_1}
to obtain that 
\begin{align*}
\sum_{i=1}^{r} i(r-1-i)a_i &\leq r(r-2)(t+k) - r(r-1)t. 
\intertext{In other words,}
\sum_{i=1}^{r-2} i(r-1-i)a_i 
&\leq r(r-2)(t+k) - r(r-1)t + ra_r\\ 
&= ra_r - rt + r(r-2)k \\
&= r(t-(r-2)k+x) - rt + r(r-2)k 
= rx.
\end{align*}
Noting that $i(r-1-i) \geq r-2$ (with equality at 
$i\in\{1,r-2\}$), the above inequality implies
\begin{equation}
\label{eq:bound_small_a}
a_1+\dots+a_{r-2} \leq \frac{r}{r-2}x \leq 3x
\end{equation}
since $r\geq 3$.
\end{itemize}

Fix $B\in\calW_k$. We let 
$$\mathvocab{b_i} := 
\text{the number of elements in }B\text{ appearing in exactly }i 
\text{ sets among } A_1,\dots,A_r.$$
In particular, $b_i\leq a_i$.
Let $\mathvocab m:=k-(a_r-b_r)$, so in particular,
$b_r=t-(r-1)k+x+m$. 
We now make several observations about the $b_i$'s.
\begin{itemize}
\item Since $|B\cap A_i|\geq t$ for each $i$,
summing this across all $i$ gives 
\begin{equation}
\label{eq:bound_b_1}b_1+2b_2+\dots+rb_r \geq rt.
\end{equation}
In particular, subtracting \eqref{eq:bound_b_1} 
from \eqref{eq:bound_a_1} gives 
$$r(a_r-b_r) + (r-1)(a_{r-1}-b_{r-1}) \leq rk,$$
implying that 
\begin{equation}
\label{eq:bound_diff_ab}
a_{r-1}-b_{r-1} \leq \frac{r}{r-1}m.
\end{equation}
\item We have $|B\cap A_2\cap A_3\cap\dots\cap A_r|\geq 
t-(r-2)k+x$ by minimality of $|A_1\cap A_2\cap\dots\cap A_r|$,
and a similar inequality holds if instead of taking $A_1$ away,
we take away any of $A_2$, $A_3$, \dots, or $A_r$.
Summing these $r$ inequalities together gives 
\begin{equation}
\label{eq:bound_b_2}
b_{r-1} + rb_r \geq rt - r(r-2)k + rx
\end{equation}
because each element appearing in $B$ and $r-1$
of the sets $A_1,\dots,A_r$ is counted once,
but each element appearing in $B\cap A_1\cap\dots\cap A_r$
is counted in all of the $r$ intersections.
\item We now give a lower bound on $m$.
Plugging in $b_r=t-(r-1)k+x+m$ in \eqref{eq:bound_b_2} gives
$$b_{r-1}\geq r(k-m).$$
However, $b_{r-1}\leq a_{r-1} \leq rk-\frac{rx}{r-1}$
by \eqref{eq:bound_A_r-1}, so we have 
\begin{equation}
\label{eq:bound_m}
m \geq \frac{x}{r-1}.
\end{equation}
\item 
We give a lower bound on $b_1+b_2+\dots+b_r$.
To do this, observe that
\begin{align*}
(r-1)(b_1+\dots+b_r) 
&\geq b_1 + 2b_2 + \dots + (r-1)b_{r-1} + (r-1)b_r \\
&\geq rt - b_r \tag{using \eqref{eq:bound_b_1}} \\
&= (r-1)(t+k)-x-m,
\end{align*}
which implies that 
\begin{equation}
\label{eq:bound_sum_b}
b_1+\dots+b_r \geq t+k-\frac{x+m}{r-1}.
\end{equation}
(In other words, $B$ contains at most $\frac{x+m}{r-1}$
elements not in $A_1\cup\dots\cup A_r$.)
\end{itemize}
We now count $|\calW_k|$. For a fixed $b_1,\dots,b_r$,
the upper bound on elements in $\calW_k$ with 
these values of $b_1,\dots,b_r$ is given by 
$$\binom{a_r}{b_r} \cdots \binom{a_1}{b_1} 
k^{t+k-b_r-\dots-b_1}$$
Indeed, for each $i$, there are $\binom{a_i}{b_i}$ ways to 
pick elements from those that appear in $i$ sets 
in $A_1,\dots,A_r$. Let $S$ be the union of all those elements.
Once we do this for all $i$, we must pick a set from 
$\calW_k(S)$, so
by \Cref{prop:peeling_properties} (d),
there are at most $|\calW_k(S)| \leq k^{t+k-|S|} 
= k^{t+k-b_r-\dots-b_1}$ ways 
to pick the remaining elements 
(regardless of the choices made earlier).
This justify the bound above.

Using the bounds $\binom{a_r}{b_r}=\binom{a_r}{a_r-b_r}$,
$\binom{a_{r-1}}{b_{r-1}} = \binom{a_{r-1}}{a_{r-1}-b_{r-1}}$,
and $\binom{a_i}{b_i}\leq 2^{a_i}$ for all $i\in\{1,2,\dots,r-2\}$,
we see that the above quantity is at most
\begin{align*}
&\binom{a_r}{a_r-b_r} \binom{a_{r-1}}{a_{r-1}-b_{r-1}}
2^{a_{r-2}+\dots+a_1} k^{t+k-b_r-\dots-b_1} \\
&\leq 
\binom{t+k}{k-m} \binom{rk}{a_{r-1}-b_{r-1}}
2^{3x} k^{(x+m)/(r-1)} \tag{using $a_r\leq t+k$, 
\eqref{eq:bound_A_r-1}, \eqref{eq:bound_small_a}, and \eqref{eq:bound_sum_b}} \\
&\leq \binom{t+k}{k-m} (rk)^{a_{r-1}-b_{r-1}} 2^{3(r-1)m} 
k^{(x+m)/(r-1)} \tag{using \eqref{eq:bound_m}}\\
&\leq \binom {t+k}k \left(\frac{k}{t}\right)^m 
(rk)^{a_{r-1}-b_{r-1}} 2^{3(r-1)m} 
k^{(x+m)/(r-1)} \tag{using \Cref{lem:binom_bound_2}} \\
&\leq \binom {t+k}k k^{\frac{x+m}{r-1} + a_{r-1}-b_{r-1} + m}
\cdot 
\left(\frac{2^{3(r-1)}r^2}{t}\right)^m \tag{using \eqref{eq:bound_diff_ab}}.
\end{align*}
Note that the exponent of $k$ is 
\begin{align*}
\frac{x+m}{r-1}+a_{r-1}-b_{r-1}+m 
&\leq \frac{x+m}{r-1}+ \frac{r}{r-1}m +m 
\tag{using \eqref{eq:bound_diff_ab}} \\
&\leq m + \frac{m}{r-1} + \frac{r}{r-1}m + m 
\tag{using \eqref{eq:bound_m}} \\
&= \frac{3r-1}{r-1}m 
\end{align*}
so using $k < t^{2/7 - \eps/2}$
and \eqref{eq:cond_r}, we get that
\begin{align*}
|\calW_k| &\leq \binom{t+k}k  \sum_{b_1,\dots,b_r} 
t^{\left(\frac 27-\frac{\eps}2\right) \cdot \frac{3r-1}{r-1}\cdot m} \cdot 
\left(\frac{2^{3(r-1)}r^2}{t}\right)^m \\
&\leq 
\binom{t+k}k  \sum_{b_1,\dots,b_r} 
\left(\frac{2^{3(r-1)}r^2}{t^{1/7+\eps}}\right)^m
\end{align*}
Next, we note that for each fixed $m$,
we have that from subtracting 
\eqref{eq:bound_b_1} from \eqref{eq:bound_a_1}
that $(r-1)(a_{r-1}-b_{r-1}) + 
\dots + 2(a_2-b_2)+(a_1-b_1)\leq rm$,
which implies that there are at most $(rm)^r$ choices 
of $(b_1,\dots,b_r)$ given $m$.
We also note from \eqref{eq:bound_m} that $m\geq 1$.
Therefore, we have 
$$|\calW_k| \leq \binom {t+k}k \sum_{m=1}^\infty 
(rm)^r \left(\frac{2^{3(r-1)}r^2}{t^{1/7+\eps}}\right)^m,$$
which is a converges to a finite constant if $t$ is large enough.
Combining with \Cref{lem:close_binomial},
we conclude that $|\calW_k| < Ct^{-1/7-\eps}\binom tk$
for some constant $C$, violating the condition 
$|\calW_k| > t^{-1/7}\binom tk$ in \eqref{eq:cond_k}.
\end{proof}
\subsection{Upper Bound on \texorpdfstring{$|\calT_k|$}{|T\_k|}}
We take $A_1,\dots,A_r$ as in \Cref{lem:good_A},
so $|A_1\cap \dots\cap A_r| = t-(r-2)k$.
Recall from the discussion right after \Cref{lem:good_A}
that any element in 
$A_1\cup \dots\cup A_r$
must appear in at least $r-1$ sets among 
$A_1$, \dots, $A_r$.
We set 
\begin{align*}
\mathvocab U &:= A_1\cap \dots \cap A_r \\
\mathvocab V &:= (A_1\cup \dots\cup A_r) \setminus U,
\end{align*}
so $V$ is the set of elements that appear exactly $r-1$ times 
in $A_1,\dots,A_r$ (by the consequence of 
\Cref{lem:good_A}).
Note that $|U|=t-(r-2)k$ and $|V|=rk$.

Having established nice sets $A_1,\dots,A_r$,
we now upper-bound the size of $|\calF[\calT_k]|$
by counting the number of sets that intersect 
$A_i$ in at least $t$ elements for each $i$.
To do this, define 
$$\mathvocab{\calG_j} := \calT_k \cap \binom{[n]^2}{t+j}
\qquad \text{for each } j\leq k,$$
i.e., $\calG_j$ consists of partial permutations in $\calT_k$
that have $t+j$ elements.
In particular, $\calG_k=\calW_k$.
The following lemma gives an upper bound the size of 
$|\calG_j|$.
\begin{lemma}
\label{lem:bound_G_j}\hspace{0pt}
Assume that $t>t_0(\eps)$ and $k>k_0(\eps)$
for some constants $t_0(\eps)$ and $k_0(\eps)$.
\begin{enumerate}[label=(\alph*)]
\item 
For any $j\leq k-1$, we have 
$|\calG_j| \leq t^{-0.1} \cdot 2^{j-k} \binom tj$.
\item 
We have $|\calW_k| = (1+o(1))\binom tk$.
Furthermore, there are at most $o\left(\binom tk\right)$ 
elements in $\calG_k=\calW_k$ that are not contained 
in $A_1\cup \dots\cup A_r$.
\end{enumerate}
\end{lemma}
\begin{proof}
We do most of our work in (a) and (b) together.
In what follows, $j$ ranges through $0,1,2,\dots,k$.

Fix an element $G\in\calG_j$.
Let $t-x=|G\cap U|$ and $y=|G\cap V|$. 
By summing $|G\cap A_i|\geq t$ across all $i$, we get that 
$$r(t-x) + (r-1)y \geq rt \implies 
y \geq \frac{rx}{r-1}.$$
In other words, $G$ must contain at least 
$\big\lceil \frac{rx}{r-1}\big\rceil$
elements from $V$. Therefore, 
we upper-bound the number of possible $G$'s by casework on $x$
and noting the following observations.
\begin{itemize}
\item There are $\binom{t-(r-2)k}{t-x}$ ways to pick $t-x$ elements from $U$
to include in $G$ (in particular, $x\geq (r-2)k$);
\item There are $\binom{rk}{\left\lceil rx/(r-1)\right\rceil}$
ways to pick $\big\lceil \frac{rx}{r-1}\big\rceil$ elements from $V$
to include in $G$; and 
\item Let $X$ be the set of 
the $t-x+\big\lceil \frac{rx}{r-1}\big\rceil
= t+\big\lceil \frac{x}{r-1}\big\rceil$ elements
picked so far. 
We claim that $|\calG_j(X)| \leq k^{t+j-|X|} = 
k^{j-\lceil x/(r-1)\rceil}$,
which will imply that there are at most $k^{j-\lceil x/(r-1)\rceil}$
ways to pick the remaining elements.

Indeed, assume that $|\calG_j(X)| > \alpha^{t+j-|X|}$
for some $\alpha>k$.
Then utilizing \Cref{obs:finding_spread},
we select maximal set $Y\supseteq X$ such that 
$|\calG_j(Y)| \geq \alpha^{|X|-|Y|}|\calG_j(X)|$.
Then $\calG_j(Y)$ is $\alpha$-spread.
Moreover,
 $|\calG_j(Y)| \geq \alpha^{|X|-|Y|}|\calG_j(X)|
 > \alpha^{t+j-|Y|} \geq 1$.
This contradicts \Cref{prop:peeling_properties} (c)
(applied on $\calY = \calG_j$
and $Y$ in place of $X$).
\end{itemize}
Therefore, we have 
$$|\calG_j| \leq  
\sum_{x=(r-2)k}^{(r-1)j} \underbrace{
\binom{t-(r-2)k}{t-x} \binom{rk}{\left\lceil rx/(r-1)\right\rceil}
k^{j-\left\lceil x/(r-1)\right\rceil}}_{=:f(x)}.$$
Let $f(x)$ denote the term inside the summation.
We note the following two observations:
\begin{enumerate}
\item First, for all $x\geq (r-2)k$, we have that
\begin{align*}
\frac{f(x+r-1)}{f(x)} 
&= \frac{\binom{t-(r-2)k}{t-x-r+1}
\binom{rk}{\lceil rx/(r-1)\rceil+r}}
{\binom{t-(r-2)k}{t-x}\binom{rk}{\lceil rx/(r-1)\rceil}}
\cdot \frac 1k\\
&\asymp_r  
\frac{t^{r-1}}{(x-(r-2)k)^{r-1}}\cdot 
\frac{\left(r\left(k-\frac{x}{r-1}\right)\right)^r}{
    \left(\frac{r}{r-1}x\right)^r}
\cdot \frac 1k \\
&\gtrsim_r
\frac{t^{r-1}}{k^{r-1}}\cdot \frac 1{k^r} 
\cdot \frac 1k  \\
&= \frac{t^{r-1}}{k^{2r}} \gg 1.
\end{align*}
(Here, the asymptotic notations are taken so that 
$r$ is fixed, but $t$ and $k$ vary. 
In other words, the hidden constants may depend on $r$.)
We deduce that when $t$ and $k$ are large,
we have $f(x+r-1) \gg f(x)$ for all $x$.
\item 
Second, if $x$ is not divisible by $r-1$, then
\begin{align*}
\frac{f(x+1)}{f(x)}
&= \frac{\binom{t-(r-2)k}{t-x-1}\binom{rk}{rx/(r-1)+1}}
{\binom{t-(r-2)k}{t-x}\binom{rk}{rx/(r-1)}}
\cdot \frac 1k\\
&\asymp_r
\frac{(t-x)}{(x+1-(r-2)k)} \cdot 
\frac{r\left(k-\frac{x}{r-1}\right)}{\frac{rx}{r-1}} 
\gtrsim_r \frac{t}{k} \cdot \frac{1}{k} 
= \frac{t}{k^2} \gg 1.
\end{align*}
\end{enumerate}
In particular, if $x_0$ is such that $f(x_0)$ is maximum,
then (2) implies that $x_0$ must be divisible by $r-1$
(because otherwise $f(x_0+1)$ has higher value),
so (1) implies that $x_0=(r-1)j$ (otherwise $f(x_0+(r-1))$
has higher value).
Thus, the term $f((r-1)j)$ is the maximum term.
Moreover, this term dominates the rest. In particular,
\begin{equation}
\label{eq:bound_G_j}
|\calG_j| \leq (1+o(1)) \binom{t-(r-2)k}{t-(r-1)j} 
\binom{rk}{rj}
\leq (1+o(1)) \binom{t}{j-(r-2)(k-j)} 
\binom{rk}{rj}
\end{equation}
for all $j\in\{0,1,2,\dots,k\}$,
where we used \Cref{lem:close_binomial}
to change $t-(r-2)k$ to $t$ without significantly
affecting the value.
Now, we prove each part of the lemma.
\begin{enumerate}[label=(\alph*)]
\item If $j\leq k-1$, note that by \Cref{lem:binom_bound_2}, we have
\begin{align*}
|\calG_j| &\leq (1+o(1)) \binom tj \cdot 
\left(\frac tj-1\right)^{-(r-2)(k-j)} 
\cdot k^{r(k-j)} \\
&\leq (1+o(1))\binom tj 
\left(\frac{k^{2r-2}}{t^{r-2}}\right)^{k-j} 
\leq t^{-0.1} 2^{j-k} \binom tj
\end{align*}
since $k \ll t^{2/7}$ and $r\geq 100$
(as chosen in \eqref{eq:cond_r}),
which implies that $k^{2r-2} \leq \frac 12 t^{r-2}$.
\item By plugging in $j=k$ in \eqref{eq:bound_G_j}, we have 
$$|\calW_k| = |\calG_k| = 
(1+o(1)) \binom tk,$$
proving the first part.
To prove the second part, we note that in the case 
where $x=(r-1)k$, we pick $j-\big\lceil \frac x{r-1}\big\rceil 
= k-k = 0$ elements outside $U$ and $V$.
In particular, the count contributed by $x=(r-1)k$
all come from sets contained in $A_1\cup\dots\cup A_r$.
As we have seen, the remaining values of $x$ contribute 
at most $o\left(\binom tk\right)$ to the sum,
so there are at most $o\left(\binom tk\right)$ elements in 
$\calW_k$ that are not contained in $A_1\cup \dots \cup A_r$,
proving the final statement.
\qedhere
\end{enumerate}
\end{proof}

The following corollary, which we will use in 
\Cref{sec:conclusion},
counts the number of permutations that contains 
some element in $\calT_k'$.
We compare this to the size of family $\calA_j$
(defined in \Cref{subsec:bound_A_k}),
which we aim to prove that it is an optimal family.
We recall from \Cref{lem:bound_A_k} that if 
$j=o(u)$, then
$$|\calA_j| < (1+o(1)) \binom{t+2j}j (n-t-j)!,$$
and from \Cref{lem:close_binomial}, if $j=o(\sqrt t)$
(which will be the case if $j\leq k < t^{2/7-\eps/2}$),
then we have $\binom{t+2j}j = (1+o(1))\binom tj$, so
$$|\calA_j| < (1+o(1)) \binom tj (n-t-j)!.$$
\begin{corollary}
\label{cor:bound_T_k}
Assume that $t>t_0(\eps)$ and $k>k_0(\eps)$
for some constants $t_0(\eps)$ and $k_0(\eps)$.
Assume also that $k=o(u)$.
Let $\calT_k'=\calT_k \setminus \binom{A_1\cup\dots\cup A_r}{t+k}$.
Then for any $\calF\subseteq \Sigma_n$, we have
$$|\calF[\calT_k']| = o\left(\max_j
|\calA_j|\right).$$
(Recall that in the max, $j$ ranges from $0$ to 
$\left\lfloor\frac{n-t}2\right\rfloor$.)
\end{corollary}
\begin{proof}
Let $\calW_k' = \calW_k \setminus \binom{A_1\cup\dots\cup A_r}{t+k}$.
Therefore, by \Cref{lem:bound_G_j} (b),
we get that $|\calW_k'| = o\left(\binom tk\right)$.
Since $\calT_k'$ is contained in $\calW_k'\cup 
\bigcup_{j=0}^{k-1}\calG_j$, we have
\begin{align*}
|\calF[\calT_k']| 
&\leq |\calF[\calW_k']| + \sum_{j=0}^{k-1} |\calF[\calG_j]| \\
&\leq (n-t-k)! |\calW_k'| + \sum_{j=0}^{k-1} (n-t-j)! |\calG_j| \\
&\leq o\left((n-t-k)! \binom tk\right) 
+ \sum_{j=0}^{k-1}  \left(\frac 12\right)^{k-j} \binom tj(n-t-j)!
\tag{using \Cref{lem:bound_G_j}} \\
&\leq o\left(|\calA_k|\right)  
+ t^{-0.1}\left(\sum_{j=0}^{k-1} \left(\frac 12\right)^{k-j}
\right) (1+o(1))|\calA_j| 
\tag{using \Cref{lem:close_binomial} and \Cref{lem:bound_A_k}}\\
&\leq o\left(|\calA_k|\right) + 
(2+o(1))t^{-0.1} \max_j |\calA_j| 
\tag{geometric series}\\ 
&= o\left(\max_j |\calA_j|\right). \qedhere
\end{align*}
\end{proof}

We close this section by discussing why we 
speculate that the argument in this section has 
a potential for improvement.
\begin{remark}
\label{rem:optimizing}
At a very high level, the argument in this section 
has two steps.
\begin{enumerate}
\item Find sets $A_1,\dots,A_r$ with a ``nice'' structure.
\item Upper bound the number of sets $B$ such that 
$|B\cap A_i| \geq t$ for all $i$.
\end{enumerate}
The key difficulty of optimizing this argument 
is to find the right definition of what constitutes
``nice'' in (1).
On one hand, we want it to be weak enough 
so that (1) works at a wide range of $k$.
On the other hand, we want it to be strong enough 
so that (2) can be carried out.
In our argument, ``nice'' means $|A_1\cap\dots\cap A_r|=t-(r-2)k$.

We believe that the condition 
$|A_1\cap\dots\cap A_r|=t-(r-2)k$ is quite strong,
to the point that it forces every element to appear in 
$A_1,\dots,A_r$ either $0$, $r-1$ or $r$ times.
In particular, we find it interesting that as 
$r$ grows larger, the argument works on a larger range 
of $k$, but the existence of such $r$ sets 
$A_1,\dots,A_r$ also guarantees the existence of $r-1$
sets (one can check that $|A_1\cap\dots\cap A_r|=t-(r-2)k$
implies $|A_1\cap \dots \cap A_{r-1}|=t-(r-3)k$).
In short, as $r$ goes larger, the argument works 
for larger range of $k$ but also produces a tuple of 
sets with a stronger condition.
We believe that the right condition should not
become strictly stronger as $r$ increase.
If we can find the right weaker condition,
then we may be able to replace the exponent $\frac 57$
with a smaller one.
\end{remark}

\section{Proof of \texorpdfstring
    {\Cref{thm:not_too_close}}{Theorem \ref{thm:not_too_close}}}
\label{sec:conclusion}
We now put everything together and prove 
\Cref{thm:not_too_close}.
We reproduce the theorem statement below.
\main*
\begin{proof}[Proof of \Cref{thm:not_too_close}]
Note that (a) follows from (b), so we focus on proving (b).
We prove the contrapositive:
let $\calF$ be a $t$-intersecting family of $\Sigma_n$
such that
$|\calF| \geq \left(1-\tfrac 1e+o(1)\right)\max_k |\calA_k|$.
We aim to show that $\calF$ is contained in 
$\sigma\calA_k\tau$ for some $\sigma,\tau\in\Sigma_n$.

If $t<n^{0.99}$, then the theorem follows from 
a previous result such as \cite{spread_approx}.
Thus, assume $t\geq n^{0.99}$.
We first apply \Cref{thm:accurate_spread_approx}
to get that there exists a set $\calS$ 
consisting of partial permutations of size at most $t+0.001u$
such that $\calS$ is $t$-intersecting 
and $|\calF\setminus\calF[\calS]| < \frac 1n\cdot u!$.
Note that since $|\calF| \geq \left(1-\tfrac 1e+o(1)\right)|\calA_0| 
= \left(1-\tfrac 1e+o(1)\right)(n-t)!>\frac 12 u!$,
we get that 
\begin{equation}
\label{eq:outside_spread_approx}
|\calF\setminus\calF[\calS]| 
< \frac 2n |\calF| = o(|\calF|)
\end{equation}

Using this $\calS$, we set up 
$\calT_j$ and $\calW_j$ for $j=q-t,q-t-1,\dots,1$
as in \Cref{subsec:peeling}.
As in \Cref{sec:bounding}, we select the largest 
$k$ such that $k\leq t^{2/7-\eps/2}$ 
and $|\calW_k| > t^{-1/7} \binom tk$.
We first deal with the main case that such $k$ 
exists, leaving the edge case for later.

\textbf{If such a $k$ exists,}
then we first note that by applying \Cref{prop:peeling_properties} (b)
repeatedly, we have 
$$\calF[\calS] = \calF[\calT_k] \cup 
\bigcup_{j=k+1}^{q-t} \calF[\calW_j].$$

Using \Cref{lem:good_A}, 
pick $A_1,\dots,A_r\in\calW_k$ such that 
$|A_1\cap\dots\cap A_r|=t-(r-2)k$.
In particular, $|A_1\cup \dots\cup A_r| = t+2k$.
We show that most elements in $\calF[\calS]$ 
must contain a $(t+k)$-element subset of $A_1\cup\dots\cup A_r$.
To do this, we bound the number of elements that 
do not satisfy this property.
\begin{itemize}
\item \textbf{Elements in $\calF[\calW_j]$ for 
$t^{2/7-\eps/2} \leq j \leq q-t$.}
Note that $q-t =  0.001(n-t)$.
In this regime, we use \Cref{cor:bound_W_k}
to get that 
\begin{align*}
|\calF[\calW_j]| &\leq \left(200\max(j,\tfrac tj)\right)^j 
\cdot (n-t-j)! \\
&\leq \left(\frac u3\right)^j 
\cdot \frac{(n-t)!}{(0.99u)^{j}} 
\tag{since $u>t^{5/7+\eps}>600\max(j,t/j)$}\\
&\leq 2^{-j} (n-t)!,
\end{align*}
which implies by summing that 
$$\sum_{j=t^{2/7-\eps/2}}^{q-t} |\calF[\calW_j]|\leq 
(n-t)! = o\left(\max_j |\calA_j|\right).$$
\item \textbf{Elements in $\calF[\calW_j]$
for $k < j <  t^{2/7-\eps/2}$.}
In this case, we use minimality of $k$ to deduce that 
$|\calW_j| < t^{-1/7} \binom tj$ for all such $j$.
Therefore, for all $k\leq j < t^{2/7-\eps/2}$, 
we have that
$$|\calF[\calW_j]| \leq t^{-1/7} \binom tj (n-t-j)!.$$
Summing and applying \Cref{lem:bound_sum_A} 
gives 
\begin{align*}
\sum_{k<j < t^{2/7-\eps/2}} |\calF[\calW_j]| 
&\leq t^{-1/7}\sum_{0\leq j < t^{2/7-\eps/2}} \binom tj (n-t-j)! \\
&\leq t^{-1/7}\sum_{0\leq j\leq u^{0.99}} \binom tj (n-t-j)! 
\tag{since $u \gg t^{5/7}$} \\
&\leq t^{-1/7} \sqrt{\frac tu} \binom t{j_0} (n-t-j_0)!
\tag{using \Cref{lem:bound_sum_A}} \\
&\leq o\left(\binom t{j_0} (n-t-j_0)!\right)
= o(|\calA_{j_0}|) = o\left(\max_j |\calA_j|\right),
\end{align*}
where the last line uses \Cref{lem:close_binomial}
and \Cref{lem:bound_A_k}.
(Note that the assumption  
$t\geq n-t$ in \Cref{lem:bound_sum_A} 
does not pose a problem 
because if $t<n-t$, the sum of $\binom tj (n-t-j)!$ 
is $0$, so the result holds regardless.)
\item \textbf{Elements of $\calF[\calT_k]$.}
From \Cref{cor:bound_T_k}, we see that the number of 
elements of $\calF[\calT_k]$ that does not contain 
a $(t+k)$-element subset of $A_1\cup\dots\cup A_r$ 
is $o\left(\max_j |\calA_j|\right)$.
\end{itemize}
Hence, the number of elements in 
$\calF[S]$ that do not contain 
a $(t+k)$-element subset of $A_1\cup\dots\cup A_r$ is at most 
$o\left(\max_j|\calA_j|\right) = o(|\calF|)$.
Combining this with \eqref{eq:outside_spread_approx}
gives that at most $o(1)$-fraction of elements in $\calF$ 
do not contain 
a $(t+k)$-element subset of $A_1\cup \dots\cup A_r$.
In other words, 
at least $1-o(1)$ fraction of elements in $\calF$ 
contain a $(t+k)$-element subset of $\mathvocab T
:=A_1\cup \dots\cup A_r$.

If $T$ contains two elements $x,y\in [n]^2$ 
with the same first coordinate,
then any element in $\calW_k$ cannot contain 
both $x$ and $y$.
Among elements in $\calW_k$ that are contained in $T$,
$\binom{t+2k-2}{t+k-2}
= \binom{t+2k-2}k$ of those contain 
both $x$ and $y$.
Combining with at most $o\left(\binom tk\right)$
elements that may not be contained in $T$, we get that 
$$|\calW_k| \leq 
\binom{t+2k}k - \binom{t+2k-2}k
+ o\left(\binom tk\right)  
= o\left(\binom tk\right)$$
In particular, by accounting for all other elements,
we have that $|\calF| = o\left(\binom tk(n-t-k)!\right)
+ o\left(\max_j |\calA_j|\right) = o\left(\max_j |\calA_j|\right)$,
which is a contradiction.

Henceforth, we assume that all elements of $T$  
have distinct first coordinates.
Similarly, assume that every element of $T$ 
has distinct second coordinates.
By relabeling elements appropriately, we assume 
$$T = \{(1,1), (2,2), \dots, (t+2k, t+2k)\},$$
so any permutation that contains a set in $\binom T{t+k}$
is in $\calA_k$.
Thus, $|\calF\setminus\calA_k| = o(|\calA_k|)$.

Now, assume that $|\calF| > \left(1-\tfrac 1e + o(1)\right) 
\max_k |\calA_k|$ and that $\calF\not\subseteq \calA_k$.
Thus, there exists a permutation 
$\sigma\in\calF$ such that $\sigma$ fixes at most
$t+k-1$ elements in $[t+2k]$.
Let $X$ be the set of fixed points in $[t+2k]$ of $\sigma$. 
We find a lot of permutations $\pi$ for which 
$|\sigma\cap\pi| \leq t-1$ as follows:
\begin{itemize}
\item Pick a set $Y\subseteq [t+2k]$ such that 
$|Y|=t+k$ and $|X\cap Y| \leq t-1$.
We claim that there are at least $\binom tk$ such sets.
To see why, we split into two cases.
\begin{itemize}
\item If $|X|\leq t-1$, then this is automatic 
regardless of what set $Y$ we pick,
so there are $\binom{t+2k}k > \binom tk$ sets.
\item If $|X|\geq t$, then we can pick $Y$
by selecting $|X|-k$ elements from set $X$
and selecting all other $t+2k-|X|$ elements not in $X$.
This gives at least $\binom{|X|}{|X|-k} \geq \binom tk$
possible sets.
\end{itemize}

\item Let $\pi$ be the permutation that fixes $Y$ 
and $\sigma(i)\neq \pi(i)$ for all $i\in [n]\setminus Y$.
There are at least $\tfrac 1e(n-t-k)!-1$ such permutations
because it corresponds to the number of derangements 
on $[n]\setminus Y$, which has size $n-t-k$.
(Recall that the number of derangements 
of $[m]$ can be counted by inclusion-exclusion 
principle to get exactly 
$m!\left(\sum_{j=0}^m \frac{(-1)^j}{j!}\right)$, which is at least $\frac{m!}e-1$.)
\end{itemize}
For all choices of $\pi$ obtained as above,
we have $|\pi\cap\sigma|\leq t-1$.
Therefore, we find at least $\left(\tfrac 1e-o(1)\right) 
\binom tk (n-t-k)! 
\geq \left(\tfrac 1e - o(1)\right)|\calA_k|$ 
permutations that intersect $\sigma$ at at most $t-1$ entries.
None of those permutations can be in $\calF$,
so $|\calF\cap\calA_k| \leq \left(1-\tfrac 1e+o(1)\right)|\calA_k|$.
Hence, $|\calF| \leq \left(1-\tfrac 1e + o(1)\right)|\calA_k|$ 
as desired.

\textbf{If there is no $k$ satisfying \eqref{eq:cond_k},}
then this means that for all $k\leq t^{2/7-\eps/2}$,
we have $|\calW_k| \geq t^{-1/7}\binom tk$.
By applying \Cref{prop:peeling_properties} (b) 
repeatedly, we have 
$$\calF[\calS] = 
\bigcup_{j=0}^{q-t} \calF[\calW_j].$$
In the case where $j\geq t^{2/7-\eps/2}$,
we use the same argument as the proof above to get that
$$\sum_{j=t^{2/7-\eps/2}}^{q-t} |\calF[\calW_j]|
= o\left(\max_j |\calA_j|\right).$$
In the case where $j<t^{2/7-\eps/2}$,
we have that $|\calW_j| < t^{-1/7}\binom tj$
for all $j < t^{2/7-\eps/2}$.
Thus, the same argument as the above proof 
(using \Cref{lem:bound_sum_A}) gives 
$$\sum_{j<t^{2/7-\eps/2}} |\calF[\calW_j]| 
= o\left(\max_j |\calA_j|\right).$$

Summing these two equations gives 
$$|\calF[\calS]| \leq 
\sum_{j=0}^{q-t} |\calF[\calW_j]|
= o\left(\max_j |\calA_j|\right),$$
and since we also have from \eqref{eq:outside_spread_approx}
that $|\calF\setminus\calF[\calS]| = o\left(\max_j|\calA_j|\right)$,
we get that $|\calF| = o\left(\max_j|\calA_j|\right)$,
which is a contradiction to what we assumed that 
$|\calF| \geq \left(1-\frac 1e+o(1)\right) 
\max_j|\calA_j|$.
\end{proof}
We make a final remark of why $\frac 57$ was chosen 
as the best exponent achievable by our technique.
\begin{remark}
\label{rem:why_57}
Suppose we want to prove the result for $t=n-n^{\alpha}$
for some $\alpha\in (0,1)$. 
\begin{itemize}
\item We will be able to use \Cref{cor:bound_W_k} 
to bound the size of $\calF[\calW_j]$
only for $j\geq n^{1-\alpha}$.
\item Thus, in \Cref{sec:bounding}, we have 
to assume that $k\leq t^{1-\alpha}$.
\item Since we are summing the upper bound of $|\calW_k|$,
we must take the largest $k$ such that 
$|\calW_k| > t^{-\frac{1-\alpha}2}\binom tk$
so that the application of \Cref{lem:bound_sum_A} works.
\item If one looks at the proof of \Cref{lem:good_A},
we see that if $A_1,\dots,A_r$ does not exist, then 
the upper bound of $|\calW_k|$ is given by 
$$|\calW_k| < \binom tk \frac{k^{\frac{3r-1}{r-1}}}t,$$
which converges to $\binom tk\cdot \frac{k^3}t$
as $r\to\infty$. (Here we take the term with $m=1$,
which dominates.)
Thus, we need 
$$\frac{k^3}t < t^{-\frac{1-\alpha}2} 
\iff t^{3(1-\alpha) - 1} < t^{-\frac{1-\alpha}2},$$
which solves to $\alpha > \frac 57$.
\end{itemize}
Therefore, $\frac 57$ is the best exponent that 
can be done by our technique.
\end{remark}

\printbibliography
\end{document}